\documentclass[reqno, english, 12pt]{amsart}
\pdfoutput = 1 

\usepackage[usenames, dvipsnames]{xcolor} 
\usepackage{amsfonts,amsmath,amssymb,amsthm,bbm,mathtools,comment}
\usepackage[shortlabels]{enumitem}

\usepackage[hyphens]{url} 
\usepackage[linktoc = all, hidelinks, colorlinks, unicode=true, pagebackref=true]{hyperref} 
\usepackage[capitalise, compress, noabbrev]{cleveref} 
\hypersetup{linkcolor={blue!70!black}, citecolor={green!70!black}, urlcolor={blue!70!black}}

\usepackage[mathscr]{euscript}
\usepackage{adjustbox,tikz,calc,graphics,babel,standalone}
\usetikzlibrary{shapes.misc,calc,intersections,patterns,decorations.pathreplacing}
\usetikzlibrary{arrows,shapes,positioning}
\usetikzlibrary{decorations.markings,quotes,arrows.meta}
\usepackage{tikz}
\usepackage[final]{microtype}
\usepackage[numbers,sort]{natbib}
\usepackage{cmtiup}
\usepackage{graphicx}
\usepackage{caption}
\usepackage{subcaption}
\usepackage{verbatim}
\usepackage{array}
\usepackage[frame,cmtip,arrow,matrix,line,graph,curve]{xy}
\usepackage{graphpap,pstricks}
\usepackage{pifont}
\usepackage{todonotes}
\tikzset{vert/.style={draw, fill=black, circle, inner sep=1pt}}

\usepackage[margin = 2.3cm,foot=1cm]{geometry}

\theoremstyle{plain}
\newtheorem{theorem}{Theorem}[section]		
\newtheorem{lemma}[theorem]{Lemma}
\newtheorem{claim}[theorem]{Claim}
\newtheorem{proposition}[theorem]{Proposition}
\newtheorem{corollary}[theorem]{Corollary}
\newtheorem{conjecture}[theorem]{Conjecture}

\newtheorem{definition}[theorem]{Definition}
\newtheorem{construction}[theorem]{Construction}
\theoremstyle{remark}
\newtheorem*{remark}{Remark}

\newcommand{\cF}{\mathcal{F}}

\newcommand{\bx}{\mathbf{x}}
\newcommand{\by}{\mathbf{y}}
\newcommand{\bz}{\mathbf{z}}

\DeclareMathOperator{\poly}{poly}

\DeclarePairedDelimiter{\abs}{\lvert}{\rvert}

\DeclarePairedDelimiter{\set}{\{}{\}}

\renewcommand{\emptyset}{\varnothing}

\renewcommand{\leq}{\leqslant}

\renewcommand{\geq}{\geqslant}

\newcommand{\eps}{\ensuremath{\varepsilon}}

\newcommand{\e}{\ensuremath{\varepsilon}}

\newcommand{\Prob}{\mathbb{P}}

\newcommand{\Ber}{\textbf{Ber}}

\newcommand\p{\mathcal{P}}
\newcommand\C{\mathscr{C}}

\crefname{conjecture}{Conjecture}{Conjectures}

\let\originalleft\left
\let\originalright\right
\renewcommand{\left}{\mathopen{}\mathclose\bgroup\originalleft}
\renewcommand{\right}{\aftergroup\egroup\originalright}

\makeatletter
\def\imod#1{\allowbreak\mkern10mu({\operator@font mod}\,\,#1)}
\makeatother

\title{Asymmetric results about graph homomorphisms}
\author{Lior Gishboliner}
\author{Eoin Hurley}

\author{Yuval Wigderson}
\thanks{LG: Department of Mathematics, University of Toronto, Canada. Email: \texttt{lior.gishboliner@utoronto.ca}. Research supported by an NSERC Discovery Grant.}
\thanks{EH: Mathematical Institute, University of Oxford, Oxford, United Kingdom. Email:
\texttt{hurley@maths.ox.ac.uk}}
\thanks{YW: Institute for Theoretical Studies, ETH Z\"urich, Z\"urich, Switzerland. Email: \texttt{yuval.wigderson @eth-its.ethz.ch}. Research supported by Dr.\ Max R\"ossler, the Walter Haefner Foundation, and the ETH Z\"urich foundation.}

\begin{document}

\begin{abstract}
Many important results in extremal graph theory can be roughly summarised as ``if a triangle-free graph $G$ has certain properties, then it has a homomorphism to a triangle-free graph $\Gamma$ of bounded size''. For example, bounds on homomorphism thresholds give such a statement if $G$ has sufficiently high minimum degree, and the approximate homomorphism theorem gives such a statement for all $G$, if one weakens the notion of homomorphism appropriately.

In this paper, we study asymmetric versions of these results, where the assumptions on $G$ and $\Gamma$ need not match. For example, we prove that if $G$ is a graph with odd girth at least $9$ and minimum degree at least $\delta \abs G$, then $G$ is homomorphic to a triangle-free graph whose size depends only on $\delta$. 
Moreover, the odd girth assumption can be weakened to odd girth at least $7$ if $G$ has bounded VC dimension or bounded domination number.
This gives a new and improved proof of a result of Huang, Liu, Rong and Xu.

We also prove that in the asymmetric approximate homomorphism theorem, the bounds exhibit
a rather surprising ``double phase transition'': the bounds are super-exponential if $G$ is only assumed to be triangle-free, they become exponential if $G$ is assumed to have odd girth $7$ or $9$, and become linear if $G$ has odd girth at least $11$. 

Our proofs use a wide variety of techniques, including entropy arguments, the Frieze--Kannan weak regularity lemma, properties of the generalised Mycielskian construction, and recent work on abundance and the asymmetric removal lemma.
\end{abstract}
\maketitle

\section{Introduction}
\subsection{Background}

In this paper, we are concerned with questions of the following type: to what extent can the triangle-freeness of a large graph be ``explained'' by a small triangle-free graph? To make this vague question more formal, let us recall that a graph $G$ is \emph{homomorphic} to a graph $\Gamma$ if there exists a function $V(G) \to V(\Gamma)$ which maps edges to edges. Equivalently, $G$ is homomorphic to $\Gamma$ if $G$ is a subgraph of a blowup\footnote{A \emph{blowup} of $\Gamma$ is obtained from $\Gamma$ by replacing every vertex by an independent set, and replacing every edge by a complete bipartite graph.} of $\Gamma$. We write $G \to \Gamma$ if $G$ is homomorphic to $\Gamma$.

Note that if $G \to \Gamma$ and $\Gamma$ is triangle-free, then certainly $G$ is triangle-free as well. If we think of $G$ as large and of $\Gamma$ as small (i.e.\ of constant size), then the existence of a homomorphism $G \to \Gamma$ gives a constant-sized ``explanation'' of the triangle-freeness of $G$. That is, $G$ is a subgraph of a blowup of the constant-sized $\Gamma$, and this structure already guarantees that $G$ is triangle-free.

There are many results which state that under certain conditions, a large triangle-free graph $G$ is homomorphic to a constant-sized triangle-free $\Gamma$. Perhaps the earliest such result is due to Andr\'asfai \cite{MR169227} (see also \cite{MR340075}), who proved that if $G$ is a triangle-free graph with minimum degree greater than $\frac 25 \abs G$, then $G$ is bipartite, i.e.\ homomorphic to $K_2$. The constant $\frac 25$ is tight in this result, as shown by the five-cycle $C_5$ as well as all of its balanced blowups. Extending Andr\'asfai's theorem, H\"aggkvist \cite{MR671908} proved that if $G$ is a triangle-free graph with minimum degree greater than $\frac 38 \abs G$, then $G$ is homomorphic to $C_5$. The constant $\frac 38$ is again best possible, as shown by the $8$-vertex M\"obius ladder $M_8$, obtained by adding the four long diagonals to $C_8$, as well as all of its balanced blowups.

That is, above minimum degree $\frac 25$, there is only one triangle-free structure, given by the graph $K_2$. At the threshold $\frac 25$, another structure appears, namely that of $C_5$. These are the only two structures above the threshold $\frac 38$, at which point a new structure $M_8$ appears. Such a pattern continues for a while, and Jin \cite{MR1264720} determined the next seven thresholds as well as the structures that appear in them. 

However, as proved by Hajnal (quoted in \cite{MR342429}), this pattern cannot continue forever. Indeed, Hajnal constructed a family of triangle-free graphs $G$ with minimum degree $(\frac 13 - o(1))\abs G$ whose chromatic number tends to infinity. In particular, this implies that these graphs $G$ cannot be homomorphic to any constant-sized triangle-free graph. Rather remarkably, \L uczak \cite{MR2260851} (extending work of Thomassen \cite{MR1956996}) proved that the constant $\frac 13$ is best possible, in the \nolinebreak following \nolinebreak sense.
\begin{theorem}[\L uczak \cite{MR2260851}]\label{thm:triangle hom threshold}
	For every $\delta>0$, there exists a triangle-free graph $\Gamma$ such that every triangle-free graph $G$ with minimum degree at least $(\frac 13 + \delta)\abs G$ is homomorphic to $\Gamma$.
\end{theorem}
More recently, Brandt and Thomass\'e \cite{brandtthomasse} gave a complete description of the family of all such graphs $\Gamma$; they are the so-called \emph{Vega graphs}.

While the results of Andr\'asfai, H\"aggkvist, Jin, \L uczak, Brandt--Thomass\'e and many others are beautiful, they are too rigid to be useful in certain applications. We would often like to assume less about the graph $G$, and are willing to obtain less precise structural information. For example, a less restrictive notion than that discussed above is that of \emph{approximate homomorphisms}. 

Given two graphs $G, \Gamma$, a function $V(G) \to V(\Gamma)$ is said to be an \emph{$\varepsilon$-approximate homomorphism} if it maps all but at most $\varepsilon \abs G^2$ edges of $G$ to edges of $\Gamma$. Equivalently, $G$ has an $\varepsilon$-approximate homomorphism into $\Gamma$ if it is $\varepsilon$-close\footnote{Two $n$-vertex graphs $G_1,G_2$ on the same vertex set are \emph{$\varepsilon$-close} if one can be obtained from the other by adding/deleting at most $\varepsilon n^2$ edges, and they are \emph{$\varepsilon$-far} otherwise.} to a graph homomorphic to $\Gamma$.

The remarkable \emph{approximate homomorphism theorem}, perhaps first explicitly observed by Tao \cite{MR2259060}, states that every triangle-free graph $G$ has an $\varepsilon$-approximate homomorphism to a triangle-free graph $\Gamma$, whose size depends only on $\varepsilon$.
In other words, up to some small noise, every large triangle-free graph $G$ can be obtained from a \emph{constant-sized} triangle-free graph $\Gamma$ by blowing it up and then passing to a subgraph, where the constant size of $\Gamma$ depends only on the amount of noise, and not on the size of $G$.

Moreover, the approximate homomorphism theorem holds much more generally, dealing not just with triangle-free graphs. To state it in full generality, let us say that 
a graph $\Gamma$ is {\em $F$-hom-free} if there is no homomorphism from $F$ to $\Gamma$. Equivalently, $\Gamma$ is $F$-hom-free if every blowup of $\Gamma$ is $F$-free. Note that if $F$ is a triangle (or more generally a clique), then the properties of $F$-hom-freeness and $F$-freeness coincide.
Clearly, if $G$ is homomorphic to an $F$-hom-free graph $\Gamma$, then $G$ is certainly $F$-hom-free as well. The general statement of the approximate homomorphism theorem gives a rough converse to this simple observation.
\begin{theorem}[Approximate homomorphism theorem]\label{thm:approx hom}
	For every graph $F$ and every $\varepsilon>0$, there exists a constant $M_F(\varepsilon)>0$ such that the following holds. If a graph $G$ is $F$-hom-free, then it has an $\varepsilon$-approximate homomorphism to an $F$-hom-free graph $\Gamma$ with $\abs \Gamma \leq M_F(\varepsilon)$.
\end{theorem}
We remark that we work with the property of $F$-hom-freeness, rather than the weaker $F$-freeness, as it is the more natural notion when working with approximate homomorphisms. But this is mostly for convenience, and it is not hard to show that \cref{thm:approx hom} is equivalent to an analogous statement about $F$-free graphs (see e.g.\ \cite{FZ}).

As we shortly discuss in greater detail, the approximate homomorphism theorem is very closely related to the famous \emph{graph removal lemma} \cite{MR519318,MR1251840,MR1404036}. Indeed, if one examines the standard proof of the graph removal lemma using Szemer\'edi's regularity lemma, one immediately sees that it naturally constructs an approximate homomorphism to an $F$-hom-free graph of constant size. We remark that the connections between the graph removal lemma and minimum degree conditions such as those discussed above have recently been studied in \cite{FW_minimum_degree,GJS_minimum_degree}.

Both \cref{thm:triangle hom threshold,thm:approx hom} are \emph{symmetric} statements, in the sense that the structural property of $G$ and $\Gamma$ are the same: they are both triangle-free in \cref{thm:triangle hom threshold}, and they are both $F$-hom-free in \cref{thm:approx hom}.
In this paper, we study \emph{asymmetric} versions of these questions, where we impose different conditions on these two graphs. Quite surprisingly, this study is very subtle, and reveals a great deal of unexpected variation. We now turn to discuss these asymmetric statements, and our main results, in more detail.

\subsection{Asymmetric homomorphism thresholds}
Given a family $\cF$ of graphs, a graph $G$ is said to be \emph{$\cF$-free} if $G$ does not contain any $F \in \cF$ as a subgraph. The \emph{homomorphism threshold} of $\cF$, denoted $\delta_{\hom}(\cF)$, is defined as the infimum of all $\delta \in [0,1]$ with the property that every $\cF$-free graph $G$ with minimum degree at least $\delta \abs G$ is homomorphic to some $\cF$-free graph $\Gamma$ whose size depends only on $\delta$. In case $\cF = \{F\}$ consists of a single graph, we write $\delta_{\hom}(F)$ instead of $\delta_{\hom}(\{F\})$. In this language, \cref{thm:triangle hom threshold} states that $\delta_{\hom}(C_3) \leq \frac 13$, and the construction of Hajnal discussed above yields a matching lower bound $\delta_{\hom}(C_3) \geq \frac 13$.

Recently, there has been a great deal of interest in determining the homomorphism thresholds of (families of) odd cycles. Let us write $\C_{2t+1}$ for the family $\{C_3,C_5,\dots,C_{2t+1}\}$ of all odd cycles of length at most $2t+1$. Letzter and Snyder \cite{MR3879964} proved that $\delta_{\hom}(\C_5) = \frac 15$, and Ebsen and Schacht \cite{MR4078811} extended this and showed that $\delta_{\hom}(\C_{2t+1})= \frac{1}{2t+1}$ for all $t$. More recently, answering a question of Ebsen and Schacht, Sankar \cite{2206.07525} proved that $\delta_{\hom}(C_{2t+1})>0$ for all $t$. This result is notable because it is the first lower bound on homomorphism thresholds that really uses the $\cF$-freeness of $\Gamma$, as opposed to relying on chromatic number lower bounds. Indeed, Thomassen \cite{MR2321926} proved that if $t \geq 2$, then every $C_{2t+1}$-free graph $G$ with minimum degree at least $\delta \abs G$ has constant chromatic number (depending only on $\delta$ and $t$). 

Our first main result concerns an asymmetric variant of homomorphism thresholds. 
\begin{theorem}\label{thm:C3C5C7}
     Let $t \geq 1$ and $\delta > 0$, and let $G$ be a $\C_{2t+5}$-free graph with minimum degree at least $\delta \abs G$. Then $G$ has a homomorphism to a $\C_{2t+1}$-free graph $\Gamma$ with $\abs \Gamma \leq (t+1)^{2/\delta}$.
\end{theorem}
For example, setting $t=1$, this result states that a graph with linear minimum degree and odd girth at least $9$ is homomorphic to a triangle-free graph of constant size.

Let us define the \emph{asymmetric homomorphism threshold} $\delta_{\hom}(\cF_1; \cF_2)$ to be the infimum of all $\delta$ such that every $\cF_1$-free graph $G$ with minimum degree at least $\delta \abs G$ is homomorphic to a constant-sized $\cF_2$-free graph $\Gamma$.
Note that if $\cF_1=\cF_2$, this precisely recovers the earlier definition of the homomorphism threshold. In this language, \cref{thm:C3C5C7} implies that $\delta_{\hom}(\mathscr C_{2t+5}; \C_{2t+1}) = \nolinebreak 0$. 

In fact, we believe that the inclusion of $C_{2t+5}$ in \cref{thm:C3C5C7} is unnecessary, and leave the following as a tantalizing open problem. 
\begin{conjecture}\label{conj:no C7}
	If $G$ is a $\C_{2t+3}$-free graph with minimum degree at least $\delta\abs G$, then $G$ is homomorphic to a $\C_{2t+1}$-free graph of constant size (depending only on $\delta$). In other words, $\delta_{\hom}(\C_{2t+3};\C_{2t+1})=0$.
\end{conjecture}
Note that, if true, \cref{conj:no C7} would be best possible, as Ebsen and Schacht \cite{MR4078811} proved that $\delta_{\hom}(\C_{2t+1};\C_{2t+1})=\frac{1}{2t+1}>0$.
As partial evidence towards \cref{conj:no C7}, we  prove the following result, which gives the same conclusion if we replace the minimum degree assumption by the assumption of bounded domination number.
\begin{theorem}\label{thm:bounded domination}
    Let $G$ be a $\C_{2t+3}$-free graph with domination number $\gamma(G)$. Then $G$ has a homomorphism to a $\C_{2t+1}$-free graph $\Gamma$ with $\abs \Gamma \leq 3\cdot (t+1)^{\gamma(G)-1}-1$.
\end{theorem}

Very recently, similar results to \cref{thm:C3C5C7,thm:bounded domination} were proved by Huang, Liu, Rong and Xu \cite{2502.09576}. In particular, one of their results \cite[Theorem 1.2]{2502.09576} establishes \cref{conj:no C7} under the added assumption that $G$ has bounded VC dimension. Here, we recall that a set $S \subseteq V(G)$ is \emph{shattered} if, for all $T \subseteq S$, there is a vertex adjacent to all vertices in $T$, but not adjacent to any vertex in $S \setminus T$, and that the \emph{VC dimension} of $G$ is defined as the maximum size of a shattered subset. VC dimension is a natural and widely studied notion of ``bounded complexity'' for a graph, and the result of Huang, Liu, Rong and Xu demonstrates that this assumption allows one to prove \cref{conj:no C7}. As it turns out, if a graph has bounded VC dimension and linear minimum degree, then it has bounded domination number (see \cref{lem:VC implies domination}), hence \cref{thm:bounded domination} gives a short alternative proof of \cite[Theorem 1.2]{2502.09576}.
\begin{corollary}\label{cor:bounded VC}
    Let $G$ be a $\C_{2t+3}$-free graph with minimum degree at least $\delta\abs G$ and VC dimension at most $d$. Then $G$ is homomorphic to a $\C_{2t+1}$-free graph $\Gamma$ with $\abs\Gamma \leq 3 \cdot (t+1)^{\frac{8d}{\delta} \log \frac{8d}{\delta}}$.
\end{corollary}
In addition to being substantially shorter, our proof of \cref{cor:bounded VC} yields much stronger quantitative bounds than the proof in \cite{2502.09576}. Indeed, the proof in \cite{2502.09576} produces such a $\Gamma$ whose size is a tower of twos of height $t+1$, and whose top-most hyperexponent is $\delta^{-O(d)}$.

\subsection{Asymmetric approximate homomorphisms}
We now turn to our study of asymmetric versions of \cref{thm:approx hom}. Here, the most fundamental question does not have to do with minimum degree restrictions, but rather with the quantitative aspects of \cref{thm:approx hom}: how large is $M_F(\varepsilon)$ as a function of $\varepsilon$?

As mentioned previously, the standard proof of \cref{thm:approx hom} follows the proof of the graph removal lemma, and these two results are in fact closely related. As such, we recall the statement of the graph removal lemma before proceeding.
\begin{theorem}[Graph removal lemma \cite{MR519318,MR1251840,MR1404036}]\label{thm:removal}
    For every graph $F$ and every $\e>0$, there exists some $\delta>0$ such that the following holds. If a graph $G$ is $\varepsilon$-far from $F$-free, then $G$ contains at least $\delta \abs G^{\abs F}$ copies of $F$.
\end{theorem}
Despite its simple statement, this is a very deep result. All known proofs are rather complex, involving tools and ideas related to Szemer\'edi's regularity lemma \cite{MR540024}. This complexity is related to our best known bounds on $\delta$ as a function of $\e$: the best known proofs of the graph removal lemma \cite{MR2811609,MR3939561,MR3156927} yield an upper bound on $1/\delta$ which is of tower type, with height $O_F(\log \frac 1 \e)$. In the other direction, it is known \cite{MR519318,MR1945375} that $1/\delta$ is at least super-polynomial in $1/\e$, namely $\frac 1 \delta \geq (\frac 1 \e)^{\Omega_F(\log \frac 1 \e)}$, whenever $F$ is non-bipartite. These upper and lower bounds are extremely far apart, and it remains a major open problem to narrow the gap.

The standard proof of \cref{thm:removal} using Szemer\'edi's regularity lemma also immediately yields \cref{thm:approx hom}. Unfortunately, this proof only supplies tower-type bounds on $M_F(\varepsilon)$ as a function of $\varepsilon$. In fact, as proved by Hoppen--Kohayakawa--Lang--Lefmann--Stagni \cite{MR4132523} (for the upper bound) and by Fox--Zhao \cite{FZ} (for the lower bound), the bounds in these two theorems are closely related. The formal results are somewhat technical, but roughly speaking, they imply that the best constant $M_F(\varepsilon)$ in \cref{thm:approx hom} is exponential in $1/\delta_F(\varepsilon)$, where $\delta_F(\varepsilon)$ is the best constant in \cref{thm:removal}. In particular, the Fox--Zhao result \cite{FZ} implies that if $F$ is not bipartite, then
\begin{equation}\label{eq:superexp}
	M_F(\varepsilon) \geq 2^{(1/\varepsilon)^{\Omega_F(\log (1/\varepsilon))}},
\end{equation}
that is, that $M_F(\varepsilon)$ is larger than any function of the form\footnote{We use the notation $\poly(\e)$ to denote any function $f$ of the form $f(\e) = c \e^C$ for some absolute constants $c,C>0$.} $2^{1/{\poly(\varepsilon)}}$.

Here we are interested in an asymmetric version of this problem. Namely, we study the following function:
\begin{definition}
    For graphs $F,H$ with $H \rightarrow F$ and for $\varepsilon > 0$, let $M_{F,H}(\varepsilon)$ be the smallest $M$ such that every $H$-hom-free graph has an $\varepsilon$-approximate homomorphism to an $F$-hom-free graph on at most $M$ vertices.
\end{definition}

The case $H=F$ corresponds to \cref{thm:approx hom}, and the existence of $M_{F,H}(\varepsilon)$ in general is an immediate consequence of \cref{thm:approx hom}.

Our results on this topic imply that the function $M_{F,H}(\varepsilon)$ can have a wide variety of interesting behaviours, even for very simple graphs $F,H$ such as odd cycles. In particular, the following is a consequence of results that we shortly discuss in greater detail.
\begin{theorem}\label{thm:odd cycles}
	Let $\ell \geq 3$ be odd. We have that
	\[
		M_{K_3, C_\ell}(\varepsilon) = 
		\begin{cases}
			\text{superexponential} & \text{if }\ell=3,\\
			2^{\poly(1/\varepsilon)} & \text{if } \ell \in \{5,7\},\\
			O(\frac 1\varepsilon) & \text{if }\ell \geq 9.
		\end{cases}
	\]
\end{theorem}
We remark that the trichotomy in \cref{thm:odd cycles} is, to us, quite unexpected. While there are a vast number of results in the literature demonstrating that different behaviors hold for triangles than for longer odd cycles, this is the first example we are aware of where there is another ``phase transition'', between cycles of length $7$ and $9$. 
Moreover, the jumps in behaviour are quite drastic, going from linear all the way exponential, and then on to superexponential.

We now discuss the statement of \cref{thm:odd cycles} in greater detail.
The case $\ell=3$ is simply the symmetric case where, as discussed in \eqref{eq:superexp}, superexponential bounds were proved by Fox--Zhao \cite{FZ}. However, for $\ell \geq 5$, these results are new. In particular, we prove both upper and lower bounds of the form $2^{\poly(1/\varepsilon)}$ when $\ell \in \{5,7\}$, and prove a linear upper bound when $\ell \geq 9$. It is not hard to prove, as we do in \cref{prop:poly LB}, that $M_{F,H}(\varepsilon)\geq \Omega((\frac 1\varepsilon)^{\frac 12})$ for all $F,H$ with $H$ non-bipartite\footnote{If $H$ is bipartite then an $H$-hom-free graph has no edges, so trivially $M_{F,H}(\varepsilon) = 1$.}, hence the linear upper bound for $\ell \geq 9$ is also close to best possible. Each of these bounds is actually a special case of a more general result, as we now discuss.

First, we remark that the upper bound on $M_{K_3,C_5}(\varepsilon)$ actually follows from \cref{thm:bounded domination}. Indeed, it is not hard to show (as we do in \cref{lem:pull out vtxs}) that from any graph $G$, we may delete at most $\varepsilon \abs G^2$ edges to obtain a graph $G'$ with domination number at most $3/\varepsilon$. In particular, this yields an $\varepsilon$-approximate homomorphism from $G$ to a subgraph which has domination number at most $3/\varepsilon$. Since $G'$ is also $C_5$-hom-free if $G$ is, we may apply \cref{thm:bounded domination} and obtain that $M_{K_3,C_5}(\varepsilon)\leq 2^{O(1/\varepsilon)}$. The same argument shows that $M_{K_3,C_7}(\varepsilon) \leq 2^{O(1/\varepsilon)}$. 

However, this argument is very special to the case of odd cycles, since our proof of \cref{thm:bounded domination} is restricted to this case. Nevertheless, all of the new bounds encapsulated in \cref{thm:odd cycles} actually hold in much greater generality than for families of odd cycles. We now turn to discuss the statements and proofs of these more general statements, each of which yields one of the new bounds in \cref{thm:odd cycles} as a special case. We begin by continuing the discussion above, and explaining an alternative proof of the exponential upper bound on $M_{K_3,C_\ell}(\varepsilon)$, 
which is far more general. The basic  observation is that $M_{F,H}(\varepsilon)$ can be upper-bounded in terms of the parameter-dependence in the asymmetric $(F,H)$ removal lemma, which we now introduce.

It is well-known and easy to see that the graph removal lemma implies an asymmetric version of the same result. It states that for any pair of graphs $(F,H)$ with $H \to F$, if a graph $G$ is $\varepsilon$-far from $F$-free, then it contains at least $\delta \abs G^{\abs H}$ copies of $H$, for some $\delta>0$ depending only on $F,H,$ and $\varepsilon$. The deduction of the asymmetric statement from \cref{thm:removal} naturally incurs the same tower-type bounds that are the best known in \cref{thm:removal}. However, a recent line of work \cite{Csaba,GSW,GHIM,2501.15861} has demonstrated that in certain cases, the asymmetric removal lemma has \emph{much} better bounds, namely in some cases polynomial bounds. Following \cite{GSW}, we say that \emph{$H$ is $F$-abundant} if there are polynomial bounds in the asymmetric $(F,H)$ removal lemma, that is, if every graph $G$ which is $\varepsilon$-far from $F$-free contains at least $\poly(\varepsilon) \abs G^{\abs H}$ copies of $H$.

With these definitions in place, we may state our bound connecting $M_{F,H}(\varepsilon)$ to the bounds in the asymmetric $(F,H)$ removal lemma. This generalises a theorem of Hoppen--Kohayakawa--Lang--Lefmann--Stagni \cite[Theorem 1.4]{MR4132523}, and follows from the same proof technique.
\begin{proposition}\label{prop:exponential approx hom}
	Let $F,H$ be graphs and $\delta:(0,1) \to (0,1)$ be a function  with the property that every graph $G$ which is $\varepsilon$-far from $F$-hom-free contains at least $\delta(\varepsilon)\cdot {\abs G}^{\abs H}$ copies of $H$.
	Then $M_{F,H}(\varepsilon) \leq 2^{K^2}$, where 
	\begin{equation}\label{eq:Frieze-Kannan m}
	K \coloneqq \frac{5e(H)(2/\e)^{e(H)}}{\delta(\varepsilon/2)}.
	\end{equation}
\end{proposition}
If $H$ is $K_k$-abundant, then the assumption of \cref{prop:exponential approx hom} holds with $\delta(\varepsilon) = \poly(\varepsilon)$, since being $K_k$-hom-free is the same as being $K_k$-free. Therefore, we obtain the following corollary of \cref{prop:exponential approx hom}.
\begin{corollary}
	If $H$ is $K_k$-abundant then $M_{K_k,H}(\varepsilon) \leq 2^{\poly(1/\varepsilon)}$.
\end{corollary}
Gishboliner, Shapira, and Wigderson \cite{GSW} proved that $C_\ell$ is $K_3$-abundant for all odd $\ell \geq 5$, hence \cref{prop:exponential approx hom} implies that $M_{K_3,C_\ell}(\varepsilon) \leq 2^{\poly(1/\varepsilon)}$, as claimed in \cref{thm:odd cycles}. Note that this argument gives a weaker bound than the one using \cref{thm:bounded domination} (although both are of the form $2^{\poly(1/\varepsilon)}$), but this argument is much more general. For example, Gir\~ao, Hurley, Illingworth, and Michel \cite{GHIM} proved that the Petersen graph $P$ is $K_3$-abundant, hence \cref{prop:exponential approx hom} also implies that $M_{K_3,P}(\varepsilon) \leq 2^{\poly(1/\varepsilon)}$.

Both \cref{thm:bounded domination,prop:exponential approx hom} only yield exponential upper bounds on $M_{F,H}(\varepsilon)$ (and only in certain cases). However, in some cases, such as for $M_{K_3,C_\ell}(\varepsilon)$ with $\ell \geq 9$, we can obtain much stronger upper bounds. To state these results in full generality, we recall the following definition. Given a graph $F$, its \emph{$2$-subdivision} $F^{\bullet\bullet}$ is obtained from $F$ by adding two new vertices on every edge. Note that $F^{\bullet \bullet}$ is homomorphic to $F$ for any graph $F$.
\begin{theorem}\label{thm:subdivision blowup polynomial bound}
	Let $F,H$ be graphs with $H \to F^{\bullet \bullet}$. Then $M_{F,H}(\varepsilon) = O(\frac 1 \varepsilon)$.
\end{theorem}
Note that $K_3^{\bullet\bullet}= C_9$, and every odd cycle $C_\ell$ with $\ell \geq 9$ is homomorphic to $C_9$. Therefore, \cref{thm:subdivision blowup polynomial bound} in particular completes the upper bounds stated in \cref{thm:odd cycles}. We remark that the proofs of both \cref{thm:subdivision blowup polynomial bound,prop:exponential approx hom} yield efficient algorithms for constructing the $F$-hom-free graph $\Gamma$ as well as the approximate homomorphism $G \to \Gamma$. 

Therefore, all that remains to discuss is the exponential lower bound in \cref{thm:odd cycles}. Namely, we need to show that for many graphs $H$, including some $F$-abundant ones, the exponential dependence in \cref{prop:exponential approx hom} is necessary. For example, we will show that this is the case for $F = K_3$ and $H \in \{C_5,C_7\}$, completing the picture in \cref{thm:odd cycles}, but again our argument holds in much greater generality. This lower bound comes from a variant of the aforementioned construction of Fox and Zhao \cite{FZ}. In fact, we conjecture that this construction characterises the cases when $M_{F,H}(\varepsilon)$ is polynomial (and that in all other cases, this construction witnesses that $M_{F,H}(\varepsilon)$ is at least of the form $2^{\poly(1/\varepsilon)}$). To state this conjecture precisely, we now define the construction. 
\begin{construction}
\label{con:HG}
Let $F$ be a graph and let $\Delta$ be its maximum degree. 
For each vertex $v\in V(F)$, fix an ordering of the edges incident to $v$ and label them $\{(v,1),\dots,(v,\deg(v))\}$.\footnote{Thus each edge receives two labels (one from each endpoint), and knowing one of them is enough to identify the edge.} 
When considering a copy of $F$ in a graph $G$, we will (implicitly) fix an isomorphism from $F$ to the copy, so that the copy is endowed with the above edge-labelling. 

Let $G$ be an $n$-vertex graph where every edge is contained in a unique copy of $F$.
Let $F_1,\dots,F_m$ be the copies of $F$ in $G$.
The graph $G^{\star}$ is constructed as follows. The vertex set\footnote{We denote a vertex of $G^{\star}$ as $(v,\bx)$, where $v \in V(G)$ and $\bx \in \{1,\dots,\Delta\}^m$.} of $G^{\star}$ is $V(G) \times \{1,\dots,\Delta\}^m$.
For every edge $e=uv \in E(G)$, let $k \in [m]$ be the unique index with $e \in E(F_k)$, and suppose that $e$ has labels $(u,i)$ and $(v,j)$ (as an edge of $F_k$). Join all vertices $(u,\bx)$ and $(v,\by)$ such that $\bx_k = i$ and $\by_k = j$.   
These are the only edges in $G^{\star}$. 
\end{construction}
The following figure shows two examples of the construction $G^{\star}$ for $F=K_3$: first, when $G$ is also a triangle, and second, when $G$ is a ``bowtie'', consisting of two triangles sharing a vertex.
\begin{figure}[ht]
\begin{center}
	\begin{tikzpicture}
		\begin{scope}
		\foreach \x in {1,2,3} {
			\foreach \y in {0,1} \node[vert] (\x\y) at (120*\x+80+20*\y:1) {};
		}
		\draw (31) -- (10) (11) -- (20) (21) -- (30);
	\end{scope}
	\begin{scope}[xshift=6cm, scale=1.3]
		\coordinate (a) at (0,0);
		\coordinate (b) at (150:2);
		\coordinate (c) at (210:2);
		\coordinate (d) at (30:2);
		\coordinate (e) at (-30:2);
		\foreach \x in {a,b,c,d,e} {
			\foreach \y in {1,...,4} {
				\node[vert] (\x\y) at ($(\x) + (90*\y-45:.25)$) {};
			}
		}
		\foreach \m in {2,3} \foreach \n in {1,2} \draw[thin] (b\m) -- (c\n);
		\foreach \m in {1,4} \foreach \n in {1,2} \draw[thin] (b\m) -- (a\n);
		\foreach \m in {3,4} \foreach \n in {3,4} \draw[thin] (c\m) -- (a\n);
		\foreach \m in {1,4} \foreach \n in {2,3} \draw[thin] (a\m) -- (d\n);
		\foreach \m in {2,3} \foreach \n in {1,2} \draw[thin] (a\m) -- (e\n);
		\foreach \m in {1,4} \foreach \n in {3,4} \draw[thin] (d\m) -- (e\n);
	\end{scope}
	\end{tikzpicture}
\end{center}
\caption{Illustration of \cref{con:HG}}\label{fig:construction}
\end{figure}

Observe that each $F$-copy in $G$ becomes a disjoint union of complete bipartite graphs in $G^{\star}$. 
Namely, for each vertex $v$ in a copy $F_k$ of $F$, the blowup set 
$\{v\} \times \{1,\dots,\Delta\}^m$ is partitioned into $\Delta$ sets (according to the $k$th coordinate of the vector), and each of these sets participates in (at most) one complete bipartite graph corresponding to an edge of $F$ (for a total of $\deg(v)$ such complete bipartite graphs). Moreover, these partitions for different $F$-copies are orthogonal, meaning that their Venn diagram has all possible regions.  
The main difference between Construction \ref{con:HG} and the construction of Fox and Zhao \cite{FZ} is that in \cite{FZ}, each blowup set is partitioned into only two parts per $F$-copy, so the bipartite graphs arising from a single $F$-copy can overlap. We would like to avoid this situation in order to further restrict the types of subgraphs that can appear in the construction.

By adapting the proof in \cite{FZ}, we will show the following. 

\begin{lemma}\label{prop:no-approx-hom}
For every graph $F$, there exists $c>0$ such that the following holds.
Let $G$ be an $n$-vertex graph where every edge is contained in a unique copy of $F$, and let $m$ denote the number of copies of $F$ in $G$.
Let $G^{\star}$ be given by \cref{con:HG}. 
Then for $\varepsilon < \frac{cm}{n^2}$, there is no $\varepsilon$-approximate homomorphism from $G^{\star}$ to an $F$-hom-free graph on at most $2^{c m/n}$ vertices.
\end{lemma}

A graph is an {\em $F$-forest} if it can be obtained from an empty graph by repeatedly adding a copy of $F$ with at least $|V(F)|-1$ new vertices (hence the new copy shares at most one vertex with the current graph). The following figure shows an example of a $K_3$-forest.
\begin{figure}[h]
\begin{center}
	\begin{tikzpicture}
		\draw (0,0) node[vert] {} -- (150:1) node[vert] {} -- (210:1) node[vert] {} -- (0,0) -- (30:1) node[vert] {} -- (-30:1) node[vert] {} -- (0,0) (30:1) -- ++(-30:1) node[vert] {} -- ++(0,1) node[vert] {} -- (30:1) -- ++(0,1) node[vert] {} -- ++(210:1) node[vert] {} -- (30:1);
	\end{tikzpicture}
    \caption{An example of a $K_3$-forest}
\end{center} 
\end{figure}
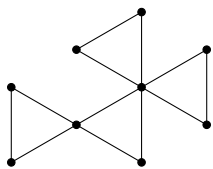

Note that if $F$ is $2$-connected then every $F$-forest $T$ satisfies the requirements of \cref{con:HG}, namely, every edge of $T$ is contained in a unique copy of $F$. Recall that $T^{\star}$ denotes the graph given by that construction.
Using \cref{prop:no-approx-hom}, we will prove the following.
\begin{theorem}\label{thm:approx hom hardness}
	Let $F,H$ be graphs, where $F$ is $2$-connected, and suppose that $H$ is not homomorphic to $T^{\star}$ for any $F$-forest $T$ with $T\rightarrow F$. 
	Then for every small enough $\varepsilon > 0$,
	$M_{F,H}(\varepsilon) \geq 2^{(1/\varepsilon)^c}$, where $c > 0$ depends only on $F,H$. 
\end{theorem}
We remark that if $F$ is vertex-transitive, then all $F$-forests are homomorphic to $F$. On the other hand, there are examples of non-vertex-transitive\footnote{Indeed, if there exist $u,v \in V(F)$ such that there exist no homomorphisms $\phi,\psi:F \to F$ with $\phi(u) = \psi(v)$, then one can construct such an $F$-forest by taking two copies of $F$ and gluing the vertex $u$ in one copy to the vertex $v$ in the other copy. This implies that if $F$ is a core (see \cite{MR1192374}), then every $F$-forest is homomorphic to $F$ if and only if $F$ is vertex-transitive.} $F$ and $F$-forests $T$ with $T \not \to F$.

As a corollary of \cref{thm:approx hom hardness} we obtain the following, implying the new lower bound in \cref{thm:odd cycles}.

\begin{corollary}\label{cor:C5 C7}
	$M_{K_3,C_\ell}(\e) \geq 2^{(1/\varepsilon)^c}$ for $\ell \in \{5,7\}$.
\end{corollary}
We remark that if $H=C_\ell$ for $\ell \geq 9$, then in fact $H \to T^\star$ for some $K_3$-forest $T$. More generally, $F^{\bullet\bullet}$ is always homomorphic to $T^\star$ for some $F$-forest $T$\footnote{Indeed, consider an $F$-forest $T$ consisting of one ``central" copy $F_0$ of $F$, and for each vertex $v$ of $F_0$, an $F$-copy $F_v$ which intersects $F_0$ at $v$. Note that $F^\star_0$ contains the graph obtained from $F_0$ by replacing each $v \in V(F_0)$ with a set $A_v$ of $\deg_F(v)$ different vertices and placing a perfect matching on $\bigcup_{v \in V(F)}A_v$ such that for every $uv \in E(F_0)$, there is an edge of the matching between $A_u$ and $A_v$. Now, for each $v \in V(F_0)$, there is a vertex $w_v \in V(T^\star)$ which is adjacent to all vertices of $A_v$ (this vertex comes from the $F$-copy $F_v$). This gives us a copy of $F^{\bullet\bullet}$ where $(w_v)_{v \in V(F)}$ play the roles of the vertices of $F$ and $\bigcup_{v \in V(F)}A_v$ play the roles of the subdivision vertices.}, and thus \cref{thm:approx hom hardness} is consistent with \cref{thm:subdivision blowup polynomial bound}. 
In fact, we conjecture that the condition in \cref{thm:approx hom hardness} characterises the cases where $M_{F,H}(\varepsilon)$ is at least $2^{\poly(1/\varepsilon)}$, and in all other cases $M_{F,H}(\varepsilon)$ \nolinebreak is \nolinebreak polynomial.
\begin{conjecture}\label{conj:approx homo characterization}
	Let $F,H$ be graphs, and suppose that there exists an $F$-forest $T$ such that $T\rightarrow F$ and $H\rightarrow T^{\star}$. Then $M_{F,H}(\varepsilon) = \poly(1/\varepsilon)$. 
\end{conjecture}
One upshot of \cref{conj:approx homo characterization} would be that \cref{con:HG} is ``universal'' for lower-bounding $M_{F,H}(\varepsilon)$, that is, that whenever this construction fails to witness an exponential lower bound, then in fact no such lower bound exists.
In particular, \cref{conj:approx homo characterization} would imply that $M_{F,H}(\varepsilon)$ cannot have any intermediate growth rate between polynomial and exponential.

\subsection{Organisation}
The rest of this paper is organised as follows. We prove our results on asymmetric homomorphism thresholds, \cref{thm:C3C5C7,thm:bounded domination}, in \cref{sec:thresholds}. Our upper bounds on $M_{F,H}(\varepsilon)$, \cref{prop:exponential approx hom,thm:subdivision blowup polynomial bound}, are proved in \cref{sec:approx hom upper}. Finally, our lower bound on $M_{F,H}(\varepsilon)$ is proved over two sections: \cref{sec:key lemma proof} contains the proof of the key \cref{prop:no-approx-hom}, and the deduction of \cref{thm:approx hom hardness} is given in \cref{sec:approx hom hard}. We end with some concluding remarks and open problems in \cref{sec:conclusion}.

All logarithms in this paper are to base $2$. We systematically omit floor and ceiling signs whenever they are not crucial.

\section{Asymmetric homomorphism thresholds}\label{sec:thresholds}
In this section, we prove \cref{thm:C3C5C7,thm:bounded domination}. 
Both are proved using the generalised Mycielskian construction, whose definition we now recall.

\begin{construction}\label{def:t-fold mycielskian}
	Given an integer $t \geq 1$ and a graph $\Gamma$, its \emph{$t$-fold Mycielskian} is the graph $M_t(\Gamma)$ defined as follows:
\begin{enumerate}
    \item $V(M_t(\Gamma)) = (V(\Gamma) \times \{1,\dots,t+1\}) \cup \set{r}$. That is, the vertices of $M_t(\Gamma)$ are of the form $(v,i)$ for some $v \in V(\Gamma)$ and $1 \leq i \leq t+1$, as well as a single distinguished vertex $r$.
    \item The induced subgraph on $V(\Gamma) \times \{t+1\}$ is a copy of $\Gamma$, and $V(\Gamma) \times \{i\}$ is an independent set for all $1 \leq i \leq t$.
    \item $r$ is adjacent to all vertices in $V(\Gamma) \times \{1\}$ and to no other vertices of $M_t(\Gamma)$.
    \item For all $1 \leq i \leq t$, each vertex $(v,i)$ is adjacent to all vertices $(w,i+1)$ with $vw \in E(\Gamma)$, and to no other vertices in $V(\Gamma) \times \{i+1\}$.
\end{enumerate}

Equivalently, one can obtain $M_t(\Gamma)$ from the tensor product of $\Gamma$ with a path on $t$ edges by inserting a copy of $\Gamma$ inside the last fibre, and adding a new vertex $r$ joined to all vertices of the first fibre.
\end{construction}
\begin{figure}[h]
    \centering
    \begin{tikzpicture}
        \node[vert] (r) at (0,0) {};
        \foreach \x in {1,...,7} {
        \foreach \y in {1,2,3} \node[vert] (\x\y) at (\x-4,-\y) {};
        \draw (r) -- (\x1);
        }
        \foreach \x in {1,...,6} {
        \pgfmathtruncatemacro{\y}{\x+1}
        \draw (\x1) -- (\y2)
        (\x2) -- (\y1)
        (\x2) -- (\y3)
        (\x3) -- (\y2)
        (\x3) -- (\y3);
        }
        \draw (11) -- (72) -- (13);
        \draw (71) -- (12) -- (73);
        \draw (13) to[out=-10, in=190] (73);
    \end{tikzpicture}
    \caption{The $2$-fold Mycielskian  $M_2(C_7)$. Note that this graph is $\C_5$-free, as implied by \cref{lem:t-fold odd girth}.}
\end{figure}
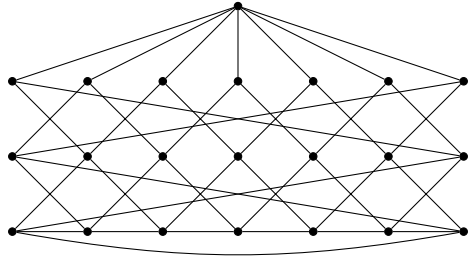

We remark that generalised Mycielskians have been well-studied in the literature, going back at least to work of Stiebitz \cite{Stiebitz}, and generalising the classical work of Mycielski \cite{MR69494}.
The key property we need about $t$-fold Mycielskians is that they preserve odd girth, as stated in the following lemma, whose proof we provide for the sake of completeness. We recall that $\C_{2t+1}$ denotes the family $\{C_3,C_5,\dots,C_{2t+1}\}$. Thus, being $\C_{2t+1}$-free is the same as having odd girth at least $2t+3$.
\begin{lemma}\label{lem:t-fold odd girth}
    If $\Gamma$ is $\C_{2t+1}$-free, then $M_t(\Gamma)$ is $\C_{2t+1}$-free as well.
\end{lemma}
Note that in case $t=1$, this just says that the standard Mycielskian construction preserves triangle-freeness, which is the basic result used by Mycielski \cite{MR69494}. 
\begin{proof}[Proof of \cref{lem:t-fold odd girth}]
    We use the following simple facts: Every closed walk of length $2t+1$ contains an odd cycle of length at most $2t+1$, and any odd cycle of length at most $2t+1$ gives rise to a closed walk of length $2t+1$ (by traversing the cycle and then going back and forth along an edge).
    Therefore, it is sufficient to prove that $M_t(\Gamma)$ contains no closed walk of length $2t+1$. Note that $M_t(\Gamma) \setminus \{r\}$ is homomorphic to $\Gamma$, via the homomorphism sending $(v,i) \mapsto v$. Therefore, as $\Gamma$ is $\C_{2t+1}$-free, any closed walk of length $2t+1$ in $M_t(\Gamma)$ must use the vertex $r$. Similarly, if we delete the layer $V(\Gamma) \times \{t+1\}$ from $M_t(\Gamma)$, we obtain a bipartite graph, with the layer $V(\Gamma) \times \{t\}$ on one side of the bipartition. As a consequence, any subgraph containing at most one vertex from $V(\Gamma) \times \{t+1\}$ is also bipartite (because in such a subgraph, the unique vertex from $V(\Gamma) \times \{t+1\}$ only has neighbours in $V(\Gamma) \times \{t\}$). Hence, any walk of length $2t+1$ must use at least two vertices of $V(\Gamma) \times \{t+1\}$. But any walk from $V(\Gamma) \times \{t+1\}$ to $r$ must use at least $t$ intermediate vertices, hence any closed odd walk in $M_t(\Gamma)$ must have length at least $2t+3$.
\end{proof}

The next lemma is the key result underpinning \cref{thm:C3C5C7,thm:bounded domination}. It allows us to extend a homomorphism $G[U] \to \Gamma$ to a homomorphism $G \to M_t(\Gamma)$, assuming that the structure of $V(G) \setminus U$ is sufficiently simple. 
By iterating this statement, we are able to prove that $G$ is homomorphic to a bounded-size graph (obtained by iterating the $t$-fold Mycielskian construction), so long as $G$ can be partitioned into ``well-behaved'' pieces. We are hopeful that this proof strategy, of using the generalized Mycielski construction to extend a homomorphism on one portion of the graph to the remainder, will have other applications besides \cref{thm:C3C5C7,thm:bounded domination}.

Given a set $X$ and an integer $i$, we denote by $N^i(X)$ the $i$th neighbourhood of $X$, namely the set of all vertices at distance $i$ from some vertex of $X$, and at distance at least $i$ from all vertices of $X$. Note that $N^0(X) = X$, $N^1(X)=N(X)$ is just the neighbourhood of $X$, and $N^i(X) \subseteq N(N^{i-1}(X))$ for all $i$.
\begin{lemma}\label{lem:generalised extend hom}
    Let $G$ be a graph, and let $U \sqcup I$ be a partition of its vertex set. Suppose that $I,N(I),N^2(I),\dots,N^t(I)$ are all independent sets in $G$. If $G[U]$ is homomorphic to some graph $\Gamma$, then $G$ is homomorphic to $M_t(\Gamma)$.
\end{lemma}
\begin{proof}
    Let $\phi:G[U] \to \Gamma$ be a homomorphism, which exists by assumption. We define $\psi:G \to M_t(\Gamma)$ as follows.
    \begin{enumerate}
        \item If $x \in I$, we set $\psi(x)=r \in V(M_t(\Gamma))$.
        \item If $x \in N^{i}(I)$ for some $1 \leq i \leq t$, then in particular $x \in U$, hence $\phi(x) = v$ for some $v \in V(\Gamma)$. We then set $\psi(x) = (v,i) \in V(M_t(\Gamma))$.
        \item Finally, if $x \in U \setminus (N^1(I) \cup \dots \cup N^t(I))$, we set $\psi(x)=(\phi(x),t+1) \in V(M_t(\Gamma))$.
    \end{enumerate}
    We claim that $\psi$ is a homomorphism. So fix any $xy \in E(G)$. As $I$ is an independent set, at most one of $x,y$ is in $I$. If $x \in I$, then the fact that $xy \in E(G)$ implies that $y \in N^1(I)$, hence $\psi(y) \in V(\Gamma) \times \{1\}$. By the definition of $M_t(\Gamma)$, we conclude that $\psi(x)=r$ is adjacent \nolinebreak to \nolinebreak $\psi(y)$.

    It thus remains to handle the case when both $x$ and $y$ are in $U$. If $x \in N^i(I)$ for some $1 \leq i \leq t$, then $y \in N^{i-1}(I) \cup N^{i+1}(I)$ (here we use the assumption that $N^i(I)$ is independent). We may assume without loss of generality that $y \in N^{i+1}(I)$ (in the other case, we interchange the roles of $x$ and $y$ and decrease $i$ by one). Let $v = \phi(x)$ and $w = \phi(y)$, which are adjacent in $\Gamma$ as $\phi$ is a homomorphism $G[U] \to \Gamma$. By the definition of $\psi$, we have $\psi(x) = (v,i)$ and $\psi(y) = (w,i+1)$, which are adjacent in $M_t(\Gamma)$ by definition. The final case is when both $x$ and $y$ are in $U \setminus (N^1(I) \cup \dots \cup N^t(I))$. In this case, again denoting $v = \phi(x), w = \phi(y)$ as above, we have that $\psi(x) = (v,t+1)$ and $\psi(y) = (w,t+1)$. But as $V(\Gamma) \times \{t+1\}$ induces a copy of $\Gamma$ in $M_t(\Gamma)$, we have that $\psi(x)$ and $\psi(y)$ are adjacent, since $v$ and $w$ are \nolinebreak adjacent \nolinebreak in \nolinebreak $\Gamma$.
\end{proof}

We are now ready to prove our results on asymmetric homomorphism thresholds, beginning with \cref{thm:C3C5C7}.

\begin{proof}[Proof of \cref{thm:C3C5C7}]
    Select a maximal collection $a_1,\dots,a_k$ of vertices in $G$ whose neighbourhoods $X_j \coloneqq N(a_j)$ are pairwise disjoint. By the minimum degree assumption, we have $\abs{X_j}\geq \delta \abs G$ for all $j$, hence $k \leq {\frac 1 \delta}$. We partition the vertices in $V(G) \setminus \bigcup_{j=1}^k X_j$ into sets $Y_1,\dots,Y_k$ as follows. For every vertex $v \in V(G) \setminus \bigcup_{j=1}^k X_j$, the maximality of the collection $a_1,\dots,a_k$ implies that $N(v) \cap X_j \neq \varnothing$ for some $j$. We then set $v \in Y_j$ for the minimal such index $j$. Note that $V(G) = X_1 \cup \dots \cup X_k \cup Y_1 \cup \dots \cup Y_k$. Next, we prove the following:

    \begin{claim}\label{claim:independent layers}
        For every $1 \leq j \leq k$ and $0 \leq i \leq t$, the sets $N^i(X_j),N^i(Y_j)$ are independent in $G$.
    \end{claim}

    \begin{proof}
    We first note that $N^i(a_j)$ is an independent set for all $1 \leq i \leq t+2$, as an edge between two vertices at distance $i$ from $a_j$ would yield a closed walk of length $2i+1 \leq 2(t+2)+1=2t+5$, a contradiction to the assumption that $G$ is $\C_{2t+5}$-free. 
    It follows that the subgraph of $G$ induced by $\bigcup_{i=0}^{t+2}N^i(a_j)$ is bipartite
    with parts $\bigcup_{i = 0 \text{ even}}^{t+2}N^i(a_j)$ and
    $\bigcup_{i = 0 \text{ odd}}^{t+2}N^i(a_j)$. Now fix any $0 \leq i \leq t$. Every vertex $v \in N^i(X_j)$ has a walk of length $i+1 < t+2$ to $a_j$, because $v$ has a path of length $i$ to $X_j = N(a_j)$. Hence, all vertices in $N^i(X_j)$ are in the same part of the bipartition, implying that $N^i(X_j)$ is independent. Similarly, every vertex $v \in N^i(Y_j)$ has a walk of length $i+2 \leq t+2$ to $a_j$, because $v$ has a path of length $i$ to $Y_j$, and every vertex in $Y_j$ has a walk of length 2 to $a_j$. Hence, again, all vertices in $N^i(Y_j)$ are in the same part of the bipartition and so $N^i(Y_j)$ is independent.   
    \end{proof}

	We now define an increasing sequence of sets $U_1 \subseteq \dots \subseteq U_{2k}$ by declaring 
	\[
		U_\ell = \bigcup_{j=1}^\ell X_j \qquad \text{ and } \qquad U_{k+\ell} = U_k \cup \bigcup_{j=1}^\ell Y_j
	\]
	for all $1 \leq \ell \leq k$. Let $G_j = G[U_j]$, and note that $G_{2k}=G$ since $U_{2k} = V(G)$. Moreover, if we set $I_{j+1} \coloneqq U_{j+1} \setminus U_{j} \in \{X_1,\dots,X_k,Y_1,\dots,Y_k\}$, then by Claim \ref{claim:independent layers}, $I_{j+1}$ and $N^i(I_{j+1})$ are independent in $G$ for all $1 \leq i \leq t$, hence also independent in $G_{j+1}$.
    
	Now let $\Gamma_1$ be the one-vertex graph, and let $\phi_1: G_1 \to \Gamma_1$ be the constant function, which is a homomorphism since $G_1$ has no edges. Inductively, having defined $\phi_j:G_j \to \Gamma_j$, we let $\Gamma_{j+1} = M_t({\Gamma_j})$, and let $\phi_{j+1}:G_{j+1} \to \Gamma_{j+1}$ be the homomorphism given by \cref{lem:generalised extend hom}, which we may apply since $V(G_{j+1}) = V(G_j) \sqcup I_{j+1}$ and since $I_{j+1}$ and $N^i(I_{j+1})$ are independent sets for all $1 \leq i \leq t$. Note that by \cref{lem:t-fold odd girth}, each of the graphs $\Gamma_j$ is $\C_{2t+1}$-free. Moreover, $\abs{\Gamma_{j}} = (t+1)\abs{\Gamma_{j-1}}+1$, hence $\abs{\Gamma_{j}} = \frac{(t+1)^{j}-1}{t}$ holds by induction. At the end of this process, we have defined a homomorphism $\phi_{2k}:G_{2k} \to \Gamma_{2k}$, where $G_{2k}=G$ and $\Gamma_{2k}$ is a $\C_{2t+1}$-free graph on $\frac{(t+1)^{2k}-1}{t} \leq (t+1)^{2k} \leq (t+1)^{2/\delta}$ vertices, as claimed. 
\end{proof}
\cref{thm:bounded domination} is proved in a very similar way; the only difference is that we do not need to pick out the maximal collection $a_1,\dots,a_k$, and instead use the given dominating set of \nolinebreak bounded \nolinebreak size.
\begin{proof}[Proof of \cref{thm:bounded domination}]
    We proceed by induction on $\gamma(G)$. The base case $\gamma(G)=1$ is trivial as then $G$ is a star, hence homomorphic to $K_2$. For the inductive step, let $a_1,\dots,a_k$ be a dominating set of $G$ with $k=\gamma(G) \geq 2$. Let $I = N(a_k)$ and let $U = V(G) \setminus I$. Note that $G_0 \coloneqq G[U \setminus \{a_k\}]$ has domination number at most $k-1$, as $a_1,\dots,a_{k-1}$ is a dominating set of $G_0$. Moreover, as a subgraph of $G$, we have that $G_0$ is $\C_{2t+3}$-free. By the inductive hypothesis, there is a $\C_{2t+1}$-free graph $\Gamma_0$ on at most $3 \cdot (t+1)^{k-2}-1$ vertices, as well as a homomorphism $\phi_0:G_0 \to \Gamma_0$. Note that $a_k$ is an isolated vertex in $G[U]$, hence we may extend $\phi_0$ to a homomorphism $\phi:G[U] \to \Gamma_0$ by mapping $a_k$ to an arbitrary vertex of $\Gamma_0$.

    As in the proof of \cref{thm:C3C5C7}, one can show that $N^i(I)$ is independent for every $0 \leq i \leq t$, using the fact that $N^i(a_k)$ is independent for every $0 \leq i \leq t+1$, which holds because $G$ is $\C_{2t+3}$-free.
    Therefore, by \cref{lem:generalised extend hom}, there is a homomorphism $\psi:G \to \Gamma$, where $\Gamma = M_t({\Gamma_0})$. Also, $\Gamma$ is $\C_{2t+1}$-free by \cref{lem:t-fold odd girth}, and 
    \[
    \abs \Gamma = (t+1)\abs{\Gamma_0}+1 \leq (t+1)(3 \cdot (t+1)^{k-2}-1) + 1 = 3\cdot (t+1)^{k-1} -t \leq 3\cdot (t+1)^{k-1}-1,
    \]
    as claimed.
\end{proof}

Finally, we turn to the proof of \cref{cor:bounded VC}, for which we need to recall a few results of VC theory. Let $\cF \subseteq 2^V$ be a set system on a finite ground set $V$. We say that $\cF$ \emph{shatters} a set $S \subseteq V$ if, for all $T \subseteq S$, there exists $F \in \cF$ with $F \cap S = T$. The \emph{VC dimension} of $\cF$ is then defined as the maximum size of a shattered set $S \subseteq V$. For a graph $G$, we define its VC dimension to be the VC dimension of the set system $\cF \coloneqq \{N(v):v \in V(G)\} \subseteq 2^{V(G)}$. Note that this agrees with the definition of VC dimension of a graph discussed in the Introduction.

The key fact that we need about set systems of bounded VC dimension is the fundamental $\varepsilon$-net theorem of Haussler and Welzl \cite{MR884223}. A set $A \subseteq V$ is called an \emph{$\varepsilon$-net} for $\cF \subseteq 2^V$ if $A$ intersects every $F \in \cF$ with $\abs F \geq \varepsilon \abs V$. Haussler and Welzl proved the following.
\begin{theorem}[{\cite[Theorem 3.8]{MR884223}}]\label{thm:eps-nets}
    If $\cF\subseteq 2^V$ has VC dimension at most $d$, then $\cF$ has an $\varepsilon$-net of size at most $\frac{8d}{\varepsilon} \log \frac{8d}{\varepsilon}$.
\end{theorem}
For our purposes, \cref{thm:eps-nets} immediately implies that graphs of bounded VC dimension and linear minimum degree have bounded domination number.
\begin{lemma}\label{lem:VC implies domination}
    If $G$ is a graph with VC dimension at most $d$ and minimum degree at least $\delta \abs G$, then the domination number of $G$ is at most $\frac{8d}{\delta} \log \frac{8d}{\delta}$.
\end{lemma}
\begin{proof}
    Let $\cF \coloneqq \{N(v): v \in V(G)\} \subseteq 2^{V(G)}$. By assumption, $\cF$ has VC dimension at most $d$. Therefore, \cref{thm:eps-nets} implies that there is a $\delta$-net $A\subseteq V(G)$ for $\cF$ of size at most $\frac{8d}{\delta} \log \frac{8d}{\delta}$. We claim that $A$ is a dominating set of $G$. Indeed, fix some $w \in V(G)$. By the minimum degree condition, we have that $\abs{N(w)} \geq \delta \abs G$, hence there must exist some $v \in A$ such that $v \in N(w)$. But this exactly means that $w$ is dominated by $A$, as claimed.
\end{proof}
\noindent
With this lemma in hand, \cref{cor:bounded VC} follows immediately.
\begin{proof}[Proof of \cref{cor:bounded VC}]
    Let $G$ be a $\C_{2t+3}$-free graph with VC dimension at most $d$ and minimum degree at least $\delta \abs G$. By \cref{lem:VC implies domination}, we have $\gamma(G) \leq \frac{8d}{\delta} \log \frac{8d}{\delta}$. Therefore, \cref{thm:bounded domination} implies that $G$ is homomorphic to a $\C_{2t+1}$-free graph $\Gamma$ with
    \[
    \abs \Gamma \leq 3 \cdot (t+1)^{\gamma(G)}\leq 3 \cdot (t+1)^{\frac{8d}{\delta} \log \frac{8d}{\delta}}.\qedhere
    \]
\end{proof}

\section{Approximate homomorphisms: upper bounds}\label{sec:approx hom upper}
In this section, we prove our upper bounds on $M_{F,H}(\varepsilon)$, beginning with \cref{prop:exponential approx hom}.
\subsection{Proof of Proposition \ref{prop:exponential approx hom}}

In the course of the proof of \cref{prop:exponential approx hom}, we will need a consequence of the Frieze--Kannan weak regularity lemma \cite{MR1723039} and the counting lemma for the cut distance. In order to state it, we will need to set up some notation. Recall that $t(H,G)$ denotes the \emph{homomorphism density} from $H$ to $G$, namely the probability that a random function $V(H) \to V(G)$ is a homomorphism. Let $W$ be a weighted graph, that is, a symmetric function $W:V \times V \to [0,1]$, where $V$ is some finite set; note that we are allowing $W$ to have ``loops'', by allowing $W(x,x)$ to be non-zero. The homomorphism density $t(H,W)$ is then defined analogously, as 
\[
t(H,W)\coloneqq \frac{1}{\abs V^{\abs H}} \sum_{v_1,\dots,v_h \in V} \prod_{ij \in E(H)} W(v_i, v_j).
\]
Equivalently, this is the expected value of $\prod_{ij \in E(H)} W(v_i,v_j)$, where $v_1,\dots,v_h$ are independent uniformly random choices of vertices from $V$. Note that this definition recovers the earlier definition of $t(H,G)$ in case $W=G$ is a graph, that is, a $\{0,1\}$-valued function. Finally, for a graph $G$ and a partition $\p = A_1 \sqcup \dotsb \sqcup A_M$ of $V(G)$, we denote by $G_\p$ the weighted graph whose vertex set is $[M]$, and where the edge weight of a pair $(i,j)$ is given by\footnote{If $A,B$ are vertex subsets of a graph $G$, we denote by $e_G(A,B)$ the number of pairs in $A \times B$ that are edges of $G$, and by $d_G(A,B) \coloneqq e_G(A,B)/(\abs A \abs B)$ the \emph{edge density}.} $d_G(A_i, A_j)$. Also, $\p$ is called \emph{equitable} if $\abs{\abs{A_i}-\abs{A_j}}\leq 1$ for all $i,j$.\footnote{To keep the presentation simple, we will always assume that $|G|$ is divisible by $M$, so that $\abs{A_i} = \abs{A_j}$ for all $i,j$. This does not affect our asymptotic results.}

The following result follows immediately by combining the statements of \cite[Lemma 9.3 and Exercise 9.7]{lovasz_book} (the Frieze--Kannan weak regularity lemma) and \cite[Lemma 10.23]{lovasz_book} (the counting lemma for the cut distance).
\begin{lemma}\label{lem:FK}
	Let $G$ be a graph and let $M\geq 1$ be an integer. There exists an equitable partition $\p$ of $V(G)$ into $M$ parts such that, for any graph $H$, we have
	\[
	\abs{t(H,G) - t(H,G_\p)} \leq \frac{4e(H)}{\sqrt{\log M}}.
	\]
\end{lemma}

\noindent
Using \cref{lem:FK}, it is straightforward to prove \cref{prop:exponential approx hom}. The idea of the proof is very simple: we will apply \cref{lem:FK} to obtain a ``small'' weighted graph $\Gamma$ that is a good approximation for $G$. If $\Gamma$ is far from being $F$-hom-free, then we may apply the assumption of \cref{prop:exponential approx hom} to conclude that it contains many copies of $H$, which, by the statement of \cref{lem:FK}, implies that $G$ itself contains many copies of $H$. As $G$ is assumed to be $H$-hom-free, this is impossible, from which we conclude that $\Gamma$ is close to being $F$-hom-free. But then by deleting a few edges from $\Gamma$ to make it $F$-hom-free, we obtain an approximate homomorphism from $G$ to a small $F$-hom-free graph, as desired.
\begin{proof}[Proof of \cref{prop:exponential approx hom}]
	Let 
	$M = 2^{K^2}$, where $K$ is as in \eqref{eq:Frieze-Kannan m},
	and let $G$ be an $n$-vertex $H$-hom-free graph. Let $\p = V_1 \sqcup \dotsb \sqcup V_M$ be the partition of $V(G)$ into $M$ parts given by \cref{lem:FK}. Let $\Gamma_0$ be the simple graph on vertex set $[M]$ whose edges are those pairs $(i,j)$ with $d_G(V_i,V_j) \geq \frac\e2$ (note that $\Gamma_0$ has no loops). Finally, let $\Gamma$ be a maximum $F$-hom-free subgraph of $\Gamma_0$. Note that there is a natural map $\phi:V(G) \to V(\Gamma_0)=V(\Gamma)$, simply mapping each vertex $v$ to the index $i$ such that $v \in V_i$. Also, $\Gamma$ is an $F$-hom-free $M$-vertex graph by construction, so it suffices to prove that $\phi$ is an $\varepsilon$-approximate homomorphism.

	To prove this, we first claim that $\Gamma_0$ is $\frac\varepsilon2$-close to $F$-hom-free. Indeed, if not, then by the definition of $\delta$, we find that $\Gamma_0$ contains at least $\delta(\frac \e 2)\cdot M^{\abs H}$ copies of $H$. Every such copy corresponds to $\abs H$ distinct vertices $v_1,\dots,v_{\abs H}$ in $V(\Gamma_0)=V(G_\p)$, such that all the $e(H)$ pairs $(v_i,v_j)$ corresponding to edges of $H$ are present in $\Gamma_0$, and thus satisfy $d_G(V_{v_i},V_{v_j}) \geq \frac \e 2$. In particular, we have
	\[
	\prod_{ij \in E(H)} G_\p(v_i,v_j) = \prod_{ij \in E(H)} d_G(V_{v_i},V_{v_j}) \geq \left(\frac \e 2\right)^{e(H)}.
	\]
	Summing this up over all the copies of $H$ in $\Gamma_0$, of which there are at least $\delta(\frac \e2)\cdot M^{\abs H}$, we conclude that
	\[
	t(H,G_\p) \geq \delta \left(\frac \e 2\right) \cdot \left(\frac \e 2\right)^{e(H)}.
	\]
	But by \cref{lem:FK}, this implies that
	\begin{align*}
	t(H,G) \geq t(H,G_\p) - \frac{4e(H)}{\sqrt{\log M}} 
	&\geq \delta \left(\frac \e 2\right)\cdot \left(\frac \e 2\right)^{e(H)} - \frac{4e(H)}{K}\\
	&=\frac 15 \delta\left(\frac \e 2\right)\cdot \left(\frac \e 2\right)^{e(H)}\\
	&>0,
	\end{align*}
	where the equality uses the choice of $K$ in \eqref{eq:Frieze-Kannan m}.
	This contradicts our assumption that $G$ is $H$-hom-free, showing that $\Gamma_0$ is, as claimed, $\frac \e2$-close to $F$-hom-free. In particular, as $\Gamma$ is a maximum $F$-hom-free subgraph of $\Gamma_0$, we conclude that
	\begin{equation}\label{eq:few edges deleted}
	e(\Gamma) \geq e(\Gamma_0) - \frac \e 2 M^2.
	\end{equation}
	We now prove that $\phi$ is an $\e$-approximate homomorphism.
	There are three types of edges of $G$ that are mapped to non-edges of $\Gamma$: those edges within a part $V_i$, those edges between parts $(V_i,V_j)$ with $d_G(V_i,V_j)<\frac\e2$, and those edges between parts $(V_i,V_j)$ such that $ij \in E(\Gamma_0) \setminus E(\Gamma)$. Since the partition is equitable, there are at most $M\binom{n/M}2 \leq \frac{n^2}{2M}$ edges of the first type. For the second type, there are at most $\binom M2$ such pairs, and each pair contributes at most $\frac\e2 (\frac nM)^2$ such edges, hence there are at most $\binom M2 \cdot \frac\e 2(\frac nM)^2 \leq \frac \e 4 n^2$ edges of the second type. Finally, for the third type, by \eqref{eq:few edges deleted}, there are at most $\frac \e2 M^2$ such pairs $(V_i,V_j)$, and each contributes at most $(\frac nM)^2$ edges, so the number of edges of the third type is at most $\frac \e 2 n^2$. In total, the number of edges of $G$ mapped to non-edges of $\Gamma$ is at most
	\[
	\frac{n^2}{2M} + \frac \e 4 n^2 + \frac \e 2 n^2 =  \left(\frac{1}{2M} + \frac{3\e}4\right)n^2 \leq \e n^2,
	\]
	since $M \geq \frac 2 {\e}$. Hence, $\phi$ is an $\e$-approximate homomorphism.
\end{proof}

\subsection{Proof of Theorem \ref{thm:subdivision blowup polynomial bound}}
In this section we prove \cref{thm:subdivision blowup polynomial bound}.
We need the following simple lemma.
\begin{lemma}\label{lem:pull out vtxs}
	Let $G$ be an $n$-vertex graph and let $\varepsilon>0$. There exist (not necessarily distinct)\footnote{By slightly modifying the proof, we could also guarantee that the vertices $v_1,\dots,v_k$ are pairwise distinct and $\{v_1,\dots,v_{k}\}$ is disjoint from $S_1,\dots,S_k$.} vertices $v_1,\dots,v_k \in V(G)$ and disjoint sets $S_1,\dots,S_k \subseteq V(G)$ with the following properties.
	\begin{enumerate}
		\item $\abs{S_i} = \varepsilon n/3$ for all $1\leq i \leq k$; hence $k \leq 3/\varepsilon$;
		\item for every $1 \leq i \leq k$, $v_i$ is adjacent to all vertices in $S_i$; and
		\item letting $X \coloneqq V(G) \setminus (S_1 \cup \dotsb \cup S_k)$, there are at most $\varepsilon n^2/2$ edges incident to $X$.
	\end{enumerate}
\end{lemma}
\begin{proof}
	We run the following greedy algorithm to construct $v_1,\dots,v_k,S_1,\dots,S_k$. 
	\begin{enumerate}
		\item Suppose we have already defined $v_1,\dots,v_\ell,S_1,\dots,S_\ell$. Let $S = S_1 \cup \dotsb \cup S_\ell$ (in case $\ell=0$, set $S=\varnothing$).\label{step:init}
		\item Let $X = V(G) \setminus S$. If there are at most $\varepsilon n^2/2$ edges incident to $X$, terminate the algorithm.
		\item If there are at least $\varepsilon n^2/6$ edges inside $X$, then there is some vertex $v \in X$ which is adjacent to at least $\varepsilon n/3$ vertices of $X$. We set $v_{\ell+1}=v$ and let $S_{\ell+1}$ be an arbitrary set of $\varepsilon n/3$ neighbours of $v$ in $X$. Then $S_{\ell+1}$ is disjoint from $S_1,\dots,S_\ell$ since $S_{\ell+1} \subseteq X = V(G) \setminus(S_1 \cup \dotsb \cup S_\ell)$.\footnote{It may be that $v$ has already appeared in the sequence $v_1,\dots,v_\ell$, which is permissible since we do not require these vertices to be distinct.} Return to step \ref{step:init}.
		\item If there are at least $\varepsilon n^2/2$ edges incident to $X$ but fewer than $\varepsilon n^2/6$ edges inside $X$, there must be at least $\varepsilon n^2/3$ edges between $S$ and $X$. Hence, there exists 
        $v \in S$ which is adjacent to at least $\varepsilon n/3$ vertices in $X$. We again let $v_{\ell+1}=v$ and let $S_{\ell+1}$ be an arbitrary set of $\varepsilon n/3$ neighbours of $v$ in $X$. Again, $S_{\ell+1}$ is disjoint from $S_1,\dots,S_\ell$. Return to step \ref{step:init}.
	\end{enumerate}
	At the end of this process, we have found vertices $v_1,\dots,v_k$ and sets $S_1,\dots,S_k$. As discussed above, the sets $S_1,\dots,S_k$ are pairwise disjoint and satisfy $\abs{S_i}=\varepsilon n/3$ for all $i$. 
    Additionally, by construction, we have that $v_i$ is adjacent to all vertices in $S_i$, and that there are at most $\varepsilon n^2/2$ edges incident to $X=V(G)\setminus (S_1\cup \dotsb \cup S_k)$, since this is precisely the condition for terminating the algorithm.
\end{proof}

\begin{remark}
    As mentioned in the introduction, \cref{lem:pull out vtxs} can be used to quickly deduce the bound $M_{K_3,C_5}(\varepsilon) \leq 2^{O(1/\varepsilon)}$ from \cref{thm:bounded domination}. Indeed, suppose that $G$ is $\C_5$-free, let $(v_i,S_i)_{i=1}^k,X$ be given by \cref{lem:pull out vtxs}, and set $G_0 = G - X$. Map $V(G)$ to $V(G_0)$ by mapping each $v \in V(G_0)$ to itself and mapping each vertex in $X$ to an arbitrary vertex in $G_0$. This is an $\varepsilon$-approximate homomorphism from $G$ to $G_0$, because only edges touching $X$ are mapped to non-edges, and there are at most $\varepsilon n^2$ edges touching $X$. Also, $G_0$ is $\C_5$-free and has domination number at most $k \leq 3/\varepsilon$. 
At this point we may apply \cref{thm:bounded domination} to $G_0$ to obtain the claimed result. 
\end{remark}

\noindent
Next, we prove \cref{thm:subdivision blowup polynomial bound}.
\begin{proof}[Proof of \cref{thm:subdivision blowup polynomial bound}]
	Let $G$ be an $n$-vertex $F^{\bullet\bullet}$-hom-free graph. We apply \cref{lem:pull out vtxs} to find vertices $v_1,\dots,v_k$ and pairwise disjoint sets $S_1,\dots,S_k$ with the properties given in \cref{lem:pull out vtxs}. Let $X = V(G) \setminus (S_1 \cup \dotsb \cup S_k)$. 

	Let $\Gamma$ be the graph with vertex set $\{x,s_1,\dots,s_k\}$ in which $x$ is an isolated vertex, and where $s_is_j \in E(\Gamma)$ if and only if $e(S_i,S_j)>0$.
	Note that $\abs \Gamma = 1+k \leq 1 + 3/\varepsilon=O(\frac 1 \varepsilon)$.

	We define a map $\phi:V(G) \to V(\Gamma)$ by mapping every vertex in $X$ to $x$ and every vertex in $S_i$ to $s_i$, for all $1 \leq i \leq k$.
	We first claim that $\phi$ is an $\e$-approximate homomorphism.

	Note that the only edges of $G$ mapped to non-edges of $\Gamma$ by $\phi$ are edges incident to $X$ and edges contained in some $S_i$. Indeed,
	every other edge of $G$ goes between two vertices $v \in S_i, v' \in S_j$ for some $1 \leq i<j\leq k$. As $vv'$ is an edge between $S_i$ and $S_j$, we see that $s_is_j \in E(\Gamma)$, hence $\phi$ maps $vv'$ to an edge of $\Gamma$. By \cref{lem:pull out vtxs}, there are at most $\varepsilon n^2/2$ edges incident to $X$. Additionally, the number of edges contained in some $S_i$ is at most
	\[
		\sum_{i=1}^k \binom{\abs{S_i}}2 \leq \sum_{i=1}^k \frac{\abs{S_i}^2}{2} = k \cdot \frac{(\varepsilon n/3)^2}{2} \leq \frac{\varepsilon n^2}{6},
	\]
	using the upper bounds on $k$ and $\abs{S_i}$ given by \cref{lem:pull out vtxs}. Since $\varepsilon n^2/2 + \varepsilon n^2/6 \leq \varepsilon n^2$, we conclude that 
	$\phi$ is indeed an $\e$-approximate homomorphism.

	It remains to check that $\Gamma$ is $F$-hom-free, so suppose for contradiction that there is a homomorphism $\psi:F \to \Gamma$. Since $x$ is an isolated vertex in $\Gamma$, we may assume that $x \notin \psi(V(F))$, since any vertex mapped to $x$ must be isolated in $F$, and hence can be mapped to some $s_i$ instead. Let the vertices of $F$ be $u_1,\dots,u_{\abs F}$, and define $\ell_1,\dots,\ell_{\abs F}$ by $\psi(u_1) = s_{\ell_1},\dots,\psi(u_{\abs F})=s_{\ell_{\abs F}}$.

	For every edge $u_i u_j \in E(F)$, we have that $\psi(u_i) \psi(u_j) \in E(\Gamma)$, hence $s_{\ell_i}$ is adjacent to $s_{\ell_j}$ in $\Gamma$. By the definition of $\Gamma$, we conclude that $G$ contains an edge between $S_{\ell_i}$ and $S_{\ell_j}$. That is, we can pick some $x_{ij} \in S_{\ell_i}, x_{ji} \in S_{\ell_j}$ such that $x_{ij} x_{ji} \in E(G)$.

	Now, we define a homomorphism $F^{\bullet \bullet} \to G$ as follows. For each vertex of $F^{\bullet \bullet}$ which corresponds to an original vertex $u_i$ of $F$, we map it to $v_{\ell_i}$. For the two new vertices added on an edge $u_{i}u_j$, we map them to $x_{ij}$ and $x_{ji}$. This is indeed a homomorphism, since $v_{\ell_i}$ is adjacent to all vertices in $S_{\ell_i}$ by construction, and $x_{ij}$ is adjacent to $x_{ji}$ by the way we picked these vertices. This contradicts the assumption that $G$ is $F^{\bullet \bullet}$-hom-free, and completes the proof.
\end{proof}

\section{Proof of Lemma \ref{prop:no-approx-hom}}\label{sec:key lemma proof}
We now turn to proving our exponential lower bound on $M_{F,H}(\varepsilon)$. This section contains the proof of the key technical lemma, \cref{prop:no-approx-hom}.

Following \cite{FZ}, the proof of \cref{prop:no-approx-hom}  uses the language of entropy. We now recall the necessary notions. We also refer the reader to \cite[Section 2]{Galvin} for additional information.
Given a discrete random variable $X$, we use $H(X)$ to denote the binary entropy of $X$. For two random variables $X,Y$, $H(X \mid Y) \coloneqq H(X,Y) - H(Y)$ is the conditional entropy of $X$ given $Y$, and $I(X; Y) \coloneqq H(X) - H(X\mid Y) = H(X) + H(Y) - H(X,Y)$ is the mutual information of $X,Y$. 
Note that $I(X;Y) = I(Y;X)$.
We will use the fact that if $X_1,\dots,X_m$ are independent random variables, then
\begin{equation}\label{eq:mutual information supermodularity}
\sum_{i=1}^m I(X_i;Y) \leq I((X_1,\dots,X_m);Y).
\end{equation}
This holds because $H(X_1,\dots,X_m) = \sum_{i=1}^m H(X_i)$ for independent $X_1,\dots,X_m$, and because
$H((X_1,\dots,X_m) \mid Y) \leq \sum_{i=1}^m H(X_i \mid Y)$, which is the subadditivity of (conditional) entropy.

We further define the function $H^{-1}(x)$ on $x\in [0,1]$ to be the unique value $p\in [0,1/2]$ such that $H(\Ber(p)) = x$.
Thus, for a Bernoulli random variable $X$ with entropy at least $h$, we have $\Prob[X=0],\Prob[X=1] \geq H^{-1}(h)$.
We need the following simple claim.

\begin{lemma}\label{lem:large entropy implies large probability of all outcomes}
	Let $Z$ be a random variable taking values in $\{1,\dots,k\}$ and suppose that $H(X) \geq \log k - \beta$. Then for every $a \in \{1,\dots,k\}$, $\mathbb{P}[X = a] \geq H^{-1}(h)$, where $h = \log(\frac{k}{k-1}) - \beta$. 
\end{lemma}
\begin{proof}
	Let $J$ be the indicator random variable of the event $X=a$. 
	We have
	$$
	H(X) = H(X,J) = H(X \mid J) + H(J),
	$$
	where the first equality is because $X$ determines $J$. Now,
	$$
	H(X \mid J) = \mathbb{P}[J = 1] \cdot H(X \mid J = 1) + 
	\mathbb{P}[J = 0] \cdot H(X \mid J = 0).
	$$
	Conditioned on $J=1$, $X$ is a constant random variable, so $H(X \mid J=1) = 0$. Also, $H(X \mid J= \nolinebreak 0) \leq \log(k-1)$ because conditioned on $J=0$, $X$ takes at most $k-1$ values. Combining all of the above, we get
	$$
	\log k - \beta \leq H(X) = H(X \mid J) + H(J) \leq H(X \mid J = 0) + H(J) \leq \log(k-1) + H(J).
	$$
	So $H(J) \geq \log(\frac{k}{k-1})-\beta=h$, which implies the required bound on $\mathbb{P}[X = a] = \mathbb{P}[J=1]$. 
\end{proof}

We now turn to proving \cref{prop:no-approx-hom}. Before doing so, let us sketch the strategy. Fix an $F$-hom-free graph $\Gamma$ with $\abs{\Gamma} < 2^{c m/n}$ (where $c$ will be chosen later) and a map $\phi:V(G^{\star})\rightarrow V(\Gamma)$. We wish to prove that $\phi$ is not an $\varepsilon$-approximate homomorphism, so our goal is to show that more than $\varepsilon \abs{G^{\star}}^2$ edges of $G^{\star}$ are not mapped to edges of $\Gamma$. 
	Recall that $\Delta$ is the maximum degree of $F$, and that we denote the vertices of $G^{\star}$ as $(v,\bx)$, where $v \in V(G)$ and $\bx\in\{1,\dots,\Delta\}^m$.

	For each $v \in V(G)$, let $\bx^v$ be a uniformly random element of $\{1,\dots,\Delta\}^m$, and let $\by^v = \phi(v,\bx^v)$; that is, $\by^v$ is the image under $\phi$ of the vertex $(v,\bx^v)$. 
	We wish to analyse the mutual information $I(\bx^v;\by^v)$ of these two random variables. 
	We will show that our bound on $\abs{\Gamma}$ implies that for all $v\in V(G)$, the mutual information $I(\bx^v;\by^v)$ is at most $O(m/n)$. 
	The number of $F$-copies in $G$ per vertex $v \in V(G)$ is at least $m/n$ on average, so this implies that the average ``mutual information per $F$-copy" in $G$ is at most some constant. 
	Finally, we show that if the ``mutual information of an $F$-copy $F_i$" in $G$ is at most some constant, then the proportion of edges in the blowup of $F_i$ which are not mapped to edges of $\Gamma$ is also some constant. 
	Calculations then show that this constant is larger than $\eps$ if $c$ is sufficiently small.

We now turn to the formal proof.

\begin{proof}[Proof of \cref{prop:no-approx-hom}]
	As discussed above, we fix any $F$-hom-free graph $\Gamma$ with $\abs{\Gamma} < 2^{c m/n}$, for a constant $c>0$ to be chosen later, as well as any map $\phi:V(G^{\star})\rightarrow V(\Gamma)$. For each $v \in V(G)$, let $\bx^v$ be a uniformly random element of $\{1,\dots,\Delta\}^m$, and let $\by^v = \phi(v,\bx^v)$.

	We first use our upper bound on $\abs{\Gamma}$ to obtain an upper bound on $I(\bx^v;\by^v)$ for every $v \in V(G)$. Indeed, we have $I(\bx^v;\by^v) = H(\by^v)- H(\by^v\mid\bx^v) = H(\by^v),$ where the first step is by definition, and the second uses that $H(\by^v\mid\bx^v) = 0$ as $\by^v$ is determined by $\bx^v$.
	   Moreover, $H(\by^v) \leq \log(|\Gamma|) < c \cdot \frac{m}{n}$, because $\by^v$ takes values in $V(\Gamma)$ and $\abs \Gamma \leq 2^{cm/n}$. Hence, for every $v \in V(G)$ we have
		\begin{equation}\label{eq:mutual_info_to_order}
		I(\bx^v;\by^v) < c \cdot \frac{m}{n}.
		\end{equation}

	We now consider the coordinates of $\bx^v$ separately, that is, we consider $I(\bx^v_k;\by^v)$ for $k \in [m]$.
	Using \eqref{eq:mutual information supermodularity}, we get
	\begin{equation}\label{eq:supermodular}
		\sum_{k\in[m]}I(\bx^v_k;\by^v) \leq I(\bx^v;\by^v).
	\end{equation}
	Finally we convert this into a bound on the mutual information {per $F$-copy}. Recall that $G$ has $m$ copies of $F$, denoted $F_1,\dots,F_m$, and each edge of $G$ is in exactly one of these copies.  
	\nolinebreak We \nolinebreak have
	\begin{equation}\label{eq:mutual_information_per_copy}
		\sum_{k\in[m]}\sum_{v\in V(F_k)}I(\bx^v_k;\by^v) \leq 
		\sum_{v\in V(G)}\sum_{k \in [m]}I(\bx^v_k;\by^v) \leq 
		\sum_{v\in V(G)}I(\bx^v;\by^v) < cm,
	\end{equation}
	where the first inequality holds simply because $V(F_k) \subseteq V(G)$ for every $k$; the second uses \eqref{eq:supermodular}; and the third uses \eqref{eq:mutual_info_to_order}. 
	By \eqref{eq:mutual_information_per_copy} and Markov's inequality, at least $\frac{m}{2}$ of the indices $k \in [m]$ satisfy
   \begin{equation}\label{eq:many_F-copies}
		\sum_{v\in V(F_k)}I(\bx^v_k;\by^v) \leq 2c.
	\end{equation}
    
    In what follows, for $k \in [m]$, we denote by $F_k^{\star}$ the subgraph of $G^{\star}$ consisting of the vertices $V(F_k) \times \{1,\dots,\Delta\}^m$ and of the edges $(u,\mathbf{x}),(v,\mathbf{y})$ with $uv \in E(F_k)$ (recall Construction \ref{con:HG}).
    An edge of $G^{\star}$ is {\em bad} if it is not mapped by $\phi$ to an edge of $\Gamma$.
	The remainder of the proof consists in showing that if $F_k$  satisfies \eqref{eq:many_F-copies}, then $F_k^{\star}$ contains many bad edges. More precisely, we prove the following:
	\begin{claim}\label{claim:bad edges}
		For $k \in [m]$, if 
		$\sum_{v\in V(F_k)}I(\bx^v_k;\by^v) \leq 2c$ then $F_k^{\star}$ contains at least $2c \big( \frac{|G^{\star}|}{n} \big)^2$ \nolinebreak bad \nolinebreak edges.
	\end{claim}

	Let us first complete the proof of the lemma using \cref{claim:bad edges}. As we saw above, at least $\frac{m}{2}$ of the indices $k \in [m]$ satisfy the condition of \cref{claim:bad edges}. For each such $k$, there are at least $2c \big( \frac{|G^{\star}|}{n} \big)^2$ bad edges in $F_k^{\star}$. 
	Since the $F$-copies in $G$ are edge-disjoint, these bad edges are different for different $F$-copies.
	Thus, in total, $G^{\star}$ has at least 
    $\frac{cm}{n^2}|G^{\star}|^2$ bad edges.
	So if $\varepsilon < \frac{cm}{n^2}$, then there are more than $\varepsilon |G^{\star}|^2$ bad edges,
	as required.

	From now on our goal is to prove \cref{claim:bad edges}. 
	So fix any $k \in [m]$ and let $v_1,\dots,v_f$ be the vertices of $F_k$.
	We assume that $V(F) = [f]$ and that $i \mapsto v_i$ is an isomorphism from $F$ to $F_k$. 
	Let $V_1,\dots,V_f$ be the blowup-sets in $G^{\star}$ corresponding to $v_1,\dots,v_f$, i.e., $V_i = \{v_i\} \times \{1,\dots,\Delta\}^m$. 
	To simplify the notation, let us set $X_i = \bx^{v_i}_k$ and $Y_i = \by^{v_i}$ for $1 \leq i \leq f$. 
	Thus $X_i$ is a uniformly random element of $\{1,\dots,\Delta\}$ and $Y_i = \phi(v_i,\bx^{v_i})$ is a vertex of $\Gamma$. With this notation, the assumption of \cref{claim:bad edges} is that
	$$
	\sum_{i = 1}^f I(X_i;Y_i) \leq 2c.
	$$
	Since mutual information is nonnegative, it follows that $I(X_i;Y_i) \leq 2c$ for every $1 \leq i \leq f$. 
	Also, by the definition of mutual information, we have
	\begin{equation}\label{eq:H(X_i,Y_i)}
	H(X_i \mid Y_i) = H(X_i) - I(X_i;Y_i) = \log(\Delta) - I(X_i;Y_i) \geq \log(\Delta) - 2c.
	\end{equation}  
	Now, recall that by definition, $H(X_i\mid Y_i) = \mathbb{E}_{y_i}\big[ H(X_i\mid Y_i = y_i) \big]$. 
	Since $H(X_i\mid Y_i = y_i) \leq \log(\Delta)$ for every $y_i$, we get by \eqref{eq:H(X_i,Y_i)} and Markov's inequality that 
	$$
	\mathbb{P}_{y_i}\Big[ H(X_i\mid Y_i = y_i) < \log(\Delta) - 4fc\Big] \leq \frac{1}{2f}.
	$$
	Hence, there is a set $R_i \subseteq V(\Gamma)$ with $\mathbb{P}[Y_i \in R_i] \geq 1 - \frac{1}{2f}$ and $H(X_i\mid Y_i = y_i) \geq \log(\Delta) - 4fc$ for every $y_i \in R_i$. By the union bound, the probability that $Y_i \in R_i$ for every $1 \leq i \leq f$ is \nolinebreak at \nolinebreak least \nolinebreak $\frac{1}{2}$.

	For $1 \leq i \leq f$ and $y_i \in V(\Gamma)$, define 
	$U_i^{y_i} \coloneqq \phi^{-1}(y_i)\cap V_i$. 
	Also, for $1 \leq i < j \leq f$ and $y_i,y_j \in V(\Gamma)$, define $\delta_{i,j}^{y_i,y_j}$ to be the proportion of the pairs in $U_i^{y_i} \times U_j^{y_j}$ which are bad edges.
	\begin{claim}
	\label{claim:one triple}
		For every (not necessarily distinct) $y_1,\dots,y_f$ with $y_i \in R_i$ ($i = 1,\dots,f$), there exists an edge $ij \in E(F)$ such that $\delta_{i,j}^{y_i,y_j} \geq H^{-1}(h)^2$, where $h \coloneqq \log(\frac{\Delta}{\Delta-1}) - 4fc$. 
	\end{claim}
	\begin{proof}
		As $\Gamma$ is $F$-hom-free, 
		there is an edge $ij \in E(F)$ such that $y_iy_j \notin E(\Gamma)$ (otherwise the map $i \mapsto y_i$ would be a homomorphism from $F$ to $\Gamma$).
		Without loss of generality suppose that $i=1,j=2$, so that $v_1v_2$ is an edge of $F_k$ and $y_1y_2 \notin E(\Gamma)$.
		Recall that each edge $uv$ of $F_k$ receives two labels of the form $(u,a),(v,b)$ (see \cref{con:HG}). So let $(1,a), (2,b)$ be the labels of the edge $v_1v_2$.

		Let us denote by $X'_i$ the random variable $X_i$ conditioned on $Y_i = y_i$.
		As $X_1,X_2$ are independent, so are $X'_1,X'_2$.
		Also, for each $\ell \in \{1,\dots,\Delta\}$, $X'_i$ takes value $\ell$ with probability 
		$$
		\frac{\abs{\{(v_i,\bx) \in U_i^{y_i} : \bx_k = \ell\}}}{\abs{U_i^{y_i}}} \; .$$ 
		Since $y_i \in R_i$, we have $H(X'_i) \geq \log(\Delta) - \beta$ for $i=1,2$, where $\beta \coloneqq 4fc$. By \cref{lem:large entropy implies large probability of all outcomes}, this implies that 
		$\Prob [ X_1'=a,X'_2=b] = \Prob[X'_1=a] \cdot \Prob[X'_2=b] \geq H^{-1} (h)^2$ where $h = \log(\frac{\Delta}{\Delta-1}) - \beta$.
		But by the definition of $G^{\star}$, if $X_1=a$ and $X_2=b$ then the vertices $(v_1,\bx^{v_1})$ and $(v_2,\bx^{v_2})$ are adjacent in $G^{\star}$. 
		It follows that at least an $H^{-1}(h)^2$-fraction of the pairs in $U_1^{y_1} \times U_2^{y_2}$ are adjacent and mapped to the non-edge $y_1y_2$ of $\Gamma$. This proves the claim.
	\end{proof}

	By \cref{claim:one triple}, we have
	\begin{equation}\label{eq:expected fraction of bad edges}
	\mathbb{E}_{y_1,\dots,y_f} 
	\left[ \sum_{ij \in E(F)} \delta_{i,j}^{y_i,y_j} \right] \geq 
	\mathbb{P}[ Y_i \in R_i\text{ for all }i] \cdot H^{-1}(h)^2 \geq 
	\frac{1}{2}H^{-1}(h)^2.
	\end{equation}
	On the other hand,
	\begin{equation}\label{eq:total expectation}
	\mathbb{E}_{y_1,\dots,y_f}\left[ \delta_{i,j}^{y_i,y_j} \right] = \mathbb{E}_{y_i,y_j}\left[ \delta_{i,j}^{y_i,y_j} \right],
	\end{equation}
	which equals the proportion of pairs in $V_i \times V_j$ which are bad edges. 
	Note that $|V_1| = \dots = |V_f| = \frac{|G^{\star}|}{n}$. 
	Combining \eqref{eq:total expectation} with \eqref{eq:expected fraction of bad edges} and using linearity of expectation, 
	we see that $F_k^{\star}$ contains at least 
	$c'( \frac{|G^{\star}|}{n} )^2$ bad edges, where
	$$c' \coloneqq \frac{1}{2}H^{-1}(h)^2 = \frac{1}{2}H^{-1}\left( \log\left(\frac{\Delta}{\Delta-1}\right) - 4fc\right)^2.$$
	Finally, observe that if $c$ is sufficiently small with respect to $F$, then $c' > 2c$. Indeed, $c'$ tends to the positive constant $\frac 12 H^{-1}(\log(\frac{\Delta}{\Delta-1}))$ as $c\to 0$, whereas $2c \to 0$ as $c \to 0$, implying that indeed $c'>2c$ for $c$ sufficiently small.
	This completes the proof of \cref{claim:bad edges} and hence of the lemma. 
	\end{proof}

\section{Approximate homomorphisms: lower bounds}\label{sec:approx hom hard}
\subsection{Proof of Theorem \ref{thm:approx hom hardness}}

	We need a well-known fact about hypergraphs of large girth. Recall that a {\em Berge cycle} in a hypergraph is a sequence of distinct vertices $v_1,\dots,v_k$ and distinct edges $e_1,\dots,e_k$ ($k \geq 2$), such that $v_i,v_{i+1} \in e_i$ for every $1 \leq i \leq k$, with indices taken modulo $k$. A hypergraph is a {\em hyperforest} if it can be obtained from an empty hypergraph by repeatedly adding an edge which intersects the current hypergraph in at most one vertex, with all other vertices being new. The following fact is well-known, but we include a proof in \cref{sec:F-tree} \nolinebreak for \nolinebreak completeness.

	\begin{lemma}\label{lem:hyperforest}
		A hypergraph $\mathcal{G}$ is a hyperforest if and only if it has no Berge cycles. 
	\end{lemma}

	The following lemma follows from a standard probabilistic deletion argument. Again, for completeness, we include a proof in \cref{sec:F-tree}. 
	\begin{lemma}\label{lem:high-girth hypergraph}
		For every $f,g \geq 2$ and every $n\geq f$, there is an $n$-vertex $f$-uniform, $f$-partite hypergraph with $\Omega(n^{1+1/g})$ edges and no Berge cycle of length at most $g$.
	\end{lemma}
	\noindent
	With these preliminaries as well as \cref{prop:no-approx-hom}, we are ready to prove \cref{thm:approx hom hardness}.
	\begin{proof}[Proof of \cref{thm:approx hom hardness}]
	Set $f \coloneqq \abs F$, $h \coloneqq \abs H$.
	Let $\mathcal{G}$ be the hypergraph given by \cref{lem:high-girth hypergraph} with parameters $g \coloneqq \max(f,h)$ and $n \coloneqq c_0\cdot (\frac{1}{\varepsilon})^{\frac{1}{1-1/g}}$, where $c_0 > 0$ is a small enough constant to be chosen later. 
	{Let $G$ be the graph obtained from $\mathcal{G}$ by identifying the parts of $\mathcal{G}$ with the vertices of $F$, and accordingly placing a copy of $F$ on each hyperedge of $\mathcal{G}$, so that $G$ is a subgraph of a blowup of $F$.} 
	Letting $m$ denote the number of $F$-copies in $G$, we have $m \geq e(\mathcal G)= \Omega(n^{1 + 1/g})$. 
	Let $c = c(F)$ be the constant given by \cref{prop:no-approx-hom}.
	Note that 
	$$
	\frac{cm}{n^2} = \frac{\Omega(n^{1+1/g})}{n^2} = \Omega(n^{-1+1/g}) = \Omega( c_0^{-1+1/g} \cdot \varepsilon ) > \varepsilon,
	$$
	provided that $c_0$ is small enough. 
    Also, by definition, every edge of $G$ is on a copy of $F$.

    Consider any set $X \subseteq V(G) = V(\mathcal{G})$ with $|X| \leq g$, and let $e_1,\dots,e_s$ be the hyperedges $e \in E(\mathcal{G})$ satisfying $|e \cap X| \geq 2$.  Consider the hypergraph 
    $\mathcal{G}_X = \{e_i \cap X : 1 \leq i \leq s\}$. Since $|X| \leq g$ and $\mathcal{G}$ has no Berge cycle of length at most $g$, we get that $\mathcal{G}_X$ has no Berge cycles. Therefore, $\mathcal{G}_X$ is a hyperforest by \cref{lem:hyperforest}. 
    In the case $|X| = f$, the graph $G[X]$ can contain a copy of $F$ only if $\mathcal{G}_X$ has a hyperedge of size $f$, because $F$ is 2-connected.
    Thus, the only $F$-copies in $G$ are those corresponding to hyperedges of $\mathcal{G}$. 
    As $\mathcal{G}$ has no Berge 2-cycles, it follows that every edge of $G$ is on a unique copy of $F$. This in turn allows us to apply \cref{con:HG}. 
    
	So let $G^{\star}$ be the graph given by \cref{con:HG}. 
	By \cref{prop:no-approx-hom}, there is no $\varepsilon$-approximate homomorphism from $G^{\star}$ to any $F$-hom-free graph on at most $2^{cm/n}$ vertices. Note that $2^{cm/n} = 2^{\Omega(n^{1/g})} \geq 2^{(1/\varepsilon)^{1/g}}$ (provided that $\varepsilon$ is small enough). 

	It remains to show that $G^{\star}$ is $H$-hom-free. So suppose for contradiction that there is a homomorphism $\psi:H \to G^\star$. 
	For a vertex $u \in V(G)$, let $V_u \coloneqq \{u\} \times \{1,\dots,\Delta\}^m$ denote the blowup set of $u$ in $G^{\star}$ (recall \cref{con:HG}).  
	Let $X$ be the set of all $u \in V(G)$ such that $V_u \cap \psi(V(H)) \neq \emptyset$; so $\abs X \leq \abs H = h \leq g$. 
    Let $F_1,\dots,F_m$ be an enumeration of the $F$-copies in $G$. We consider the $F$-copies in $G$ which intersect $X$ in at least 2 vertices; without loss of generality, these are 
   $F_1,\dots,F_s$. As we saw above, the hypergraph 
    $\{V(F_i) \cap X : 1 \leq i \leq s\}$ is a hyperforest. 
    Let $T$ be the $F$-forest defined as follows: $T$ has $s$ copies of $F$, denoted $F'_1,\dots,F'_s$; $V(T)$ consists of the set $X$ and the sets $V(F'_i) \setminus X$ (for $i = 1,\dots,s$), which are pairwise-disjoint; and $T[V(F'_i) \cap X] = G[V(F_i) \cap X]$ as labelled graphs.\footnote{Namely, the sets $V(F_i)\setminus X$ may intersect (in $G$), but we make these sets disjoint in $T$ to make sure that $T$ is an $F$-forest. We could avoid this technicality by increasing $g$ to $\binom{h}{2}(f-2)$, which would ensure that $G[\bigcup_{i=1}^s V(F_i)]$ is an $F$-forest.}
    Then $T$ is homomorphic to $G$ via the homomorphism which maps $F'_i$ isomorphically to $F_i$ for every $i \in [s]$. As $G$ is homomorphic to $F$, we get that $T$ is homomorphic to $F$.
    
	Let $T^{\star}$ be the graph obtained by applying \cref{con:HG} to $T$. For a vertex $(v,\bx) \in V(G^{\star})$, let $\phi(v,\bx)$ be the result of projecting $\bx$ on the first $s$ coordinates, i.e., if $\bx = (\bx_1,\dots,\bx_m)$, then $\phi(v,\bx) = (v,(\bx_1,\dots,\bx_s))$. If $\mathbf{v} \in \psi(V(H))$ then $v \in X$, so $\phi(v,\bx) \in V(T^{\star})$. We claim that $(\phi \circ \psi) : V(H) \rightarrow V(T^{\star})$ is a homomorphism. Indeed, let $u_0,v_0$ be adjacent vertices in $H$, and let $(u,\bx)=\psi(u_0)$ and $(v,\by) = \psi(v_0)$. These two vertices are adjacent in $G^\star$ since $\psi$ is a homomorphism. Hence, $uv \in E(G)$ by the definition of $G^{\star}$. Also, $u,v \in X$, so there exists 
    $1 \leq k \leq s$ such that $uv \in E(F_k)$. Let $(u,i),(v,j)$ be the labels of the edge $uv$ in $F_k$. By the definition of $G^{\star}$, the fact that $(u,\bx)$ and $(v,\by)$ are adjacent means that $\bx_k = i$ and $\by_k = j$. But $T^{\star}$ has the same adjacency criterion, meaning that $\phi(u,\bx)$ and $\phi(v,\by)$ are adjacent in $T^{\star}$. Since this holds for arbitrary $u_0 v_0 \in E(H)$, we conclude that $\phi \circ \psi$ is indeed a homomorphism $H \to T^\star$. 
    Thus, $H$ is homomorphic to $T^\star$ for an $F$-forest $T$ with $T \rightarrow F$.
    This contradicts the assumption of the theorem and concludes the proof.
	\end{proof}

	We remark that by blowing up the graph $G^\star$ constructed in the above proof, we can obtain arbitrarily large $H$-hom-free graphs which do not have an $\e$-approximate homomorphism to an $F$-hom-free graph on at most $2^{(1/\e)^c}$ vertices.
	\begin{proof}[Proof of \cref{cor:C5 C7}]
	By \cref{thm:approx hom hardness}, it is enough to prove that $C_5,C_7$ are not homomorphic to $T^{\star}$ for any $K_3$-forest $T$. 
    So fix a $K_3$-forest $T$, and let us show that $T^{\star}$ has no odd cycles of length at most $7$. This suffices because the homomorphic image of an odd cycle must contain an odd cycle. In this proof, we will use the notation introduced in \cref{con:HG}.
    Let $F_1,\dots,F_m$ be an enumeration of the triangles in $T$.
	As before, for $u \in V(T)$, we denote by $V_u := \{u\} \times \{1,2\}^m$ the corresponding blowup-set in $T^{\star}$. Let $\phi : V(T^{\star}) \rightarrow V(T)$ be the homomorphism sending $V_u$ to $u$ for every $u \in V(T)$. Let $C$ be an odd cycle in $T^{\star}$. 
	The image $\phi(C)$ must contain an odd cycle, and this cycle must be a triangle because the only odd cycles in a $K_3$-forest are triangles. So there are three vertices of $C$ which are mapped by $\phi$ to a triangle $u,v,w$ of $T$. Let us denote these three vertices by $\mathbf{u} = (u,\bx),\mathbf{v} = (v,\by),\mathbf{w} = (w,\bz)$. Without loss of generality, $F_1 = (u,v,w)$. Observe that no two vertices among $\mathbf{u},\mathbf{v},\mathbf{w}$ have a path of length 2 
	(in $T^{\star}$) between them. Indeed, if such a path goes outside $V_u \cup V_v \cup V_w$ then it certainly has length more than 2, and there is no such path inside $V_u \cup V_v \cup V_w$ (see \cref{fig:construction}).
	Also, we claim that if (say) $\mathbf{u},\mathbf{v}$ are adjacent, then the distance of $\mathbf{w}$ to each of $\mathbf{u},\mathbf{v}$ is at least $3$, and moreover, the distance of $\mathbf{w}$ to one of these vertices is at least $4$.  To see this, suppose (without loss of generality) that the edge $uv$ gets labels $(u,1),(v,2)$, the edge $vw$ gets labels $(v,1),(w,2)$, and the edge $wu$ gets labels $(w,1),(u,2)$. Then $\bx_1 = 1$, $\by_1 = 2$. Now, $\mathbf{w}$ is not adjacent to $\mathbf{u}$ or $\mathbf{v}$ (because this would require $\bx_1 = 2$ or $\by_1 = 1$, respectively), and we already saw that there is no path of length two from $\mathbf{w}$ to $\mathbf{u}$ or $\mathbf{v}$, so $\mathbf{w}$'s distance to $\mathbf{u}$ and $\mathbf{v}$ is at least $3$. Now suppose without loss of generality that $\bz_1 = 1$. We claim that $\text{dist}(\mathbf{v},\mathbf{w}) \geq 4$. Letting $V_u^2 \coloneqq \{(u,\bx') : \bx'_1 = 2\}$ and $V_v^1 \coloneqq \{ (v,\by') : \by'_1 = 1 \}$,
	observe that the set $S \coloneqq V_u^2 \cup V_v^1$ separates $\mathbf{w}$ from $\mathbf{v}$, because $S$ contains all vertices in $V_u \cup V_v$ which have a neighbour in $V_w$. Also, every vertex in $V_v^1$ is at distance at least $3$ from $\mathbf{w}$ and at least $2$ from $\mathbf{v}$, and every vertex in $V_u^2$ is at distance at least $1$ from $\mathbf{w}$ and at least $3$ from $\mathbf{v}$. It follows that $\text{dist}(\mathbf{v},\mathbf{w}) \geq 4$.

	Summarising, either every pair among $\mathbf{u},\mathbf{v},\mathbf{w}$ is at distance at least $3$, or two of these vertices are adjacent and the last vertex is at distance at least $3$ from both of them and at least $4$ from one of them. In either case $\abs C \geq 9$ or $\abs C \geq 1+3+4 = 8$, so $\abs C > 7$.
	\end{proof}

\subsection{A polynomial lower bound}\label{sec:poly lower bound}
In this section, we complete the picture in \cref{thm:odd cycles} by proving the following simple lower bound on $M_{F,H}(\varepsilon)$.
\begin{proposition}\label{prop:poly LB}
	Let $F,H$ be graphs with $H \to F$ such that $H$ is not bipartite. We have $M_{F,H}(\varepsilon) \geq \Omega((\frac{1}{\varepsilon})^{1/2+\alpha})$, where $\alpha > 0$ depends only on $F,H$.
\end{proposition}
\begin{proof}
	Let $g$ be the shortest length of an odd cycle in $H$, let $n = \frac{1}{\varepsilon}$, and let $p = n^{-1+1/g}/(4g)$. Let
	$G$ be an $n$-vertex random graph with edge probability $p$. For every $3 \leq k \leq g$, the expected number of $k$-cycles in $G$ is at most $p^k n^k \leq n/(4g)$. Hence, the total number of cycles of length at most $g$ is, in expectation, at most $n/4$. By Markov's inequality, we conclude that $G$ has at most $n/2$ cycles of length at most $g$ with probability at least $1/2$. So by deleting one vertex from every cycle of length at most $g$, we obtain a subgraph $G'$ with at least $n/2$ vertices and with girth at least $g+1$. Then $G'$ is $H$-hom-free, because every homomorphic image of $H$ contains an odd cycle of length at most $g$.

	Hence, it suffices to show that with positive probability, no induced subgraph of $G$ on $n/2$ vertices has an $\varepsilon$-approximate homomorphism to an $F$-hom-free graph on $k$ vertices, where 
	$k \coloneqq c\sqrt{p}n = \Omega(n^{\frac{1}{2} + \frac{1}{2g}}) \geq \Omega((\frac{1}{\varepsilon})^{\frac{1}{2} + \frac{1}{2g}})$, where $c > 0$ is a small enough constant.
	So fix $U \subseteq V(G)$ of size $|U| = n/2$,
	fix an $F$-hom-free graph $\Gamma_0$ on $k$ vertices, and let $\phi : U \rightarrow V(\Gamma_0)$ be a map. 
	Let $G'$ be the graph on $U$ consisting of all pairs $uv \in \binom{U}{2}$ with $\phi(u)\phi(v) \in E(\Gamma_0)$. 
	Then $\phi$ is a homomorphism from $G'$ to $\Gamma_0$. Hence, $G'$ is $F$-hom-free because $\Gamma_0$ is. In particular, $G'$ is $K_f$-free, so by Tur\'an's theorem, 
	$e(G') \leq \frac{f-2}{f-1}\frac{|U|^2}{2}$. 
	It follows that there are at least 
	$$
	\binom{|U|}{2} - \frac{f-2}{f-1}\frac{|U|^2}{2} =
	\frac{|U|^2}{2(f-1)} - \frac{|U|}{2} \geq 
	\frac{|U|^2}{2f} = \frac{n^2}{8f}
	$$ pairs 
	$uv \in \binom{U}{2}$
	which are not mapped by $\phi$ to an edge of $\Gamma_0$. Let $Z$ be the number of these pairs which are edges in $G$, so that $Z$ stochastically dominates 
	$\text{Bin}(\frac{n^2}{8f},p)$. By the Chernoff bound,
	$$
	\mathbb{P}\left[ Z \leq \frac{1}{2}\mathbb{E}[Z] \right] \leq e^{-\frac{\mathbb{E}[Z]}{8}} \leq e^{-\frac{pn^2}{64f}} = 
	e^{-\Omega(n^{1+1/g})}. 
	$$
	There are at most $2^n$ choices for $U$, at most $2^{\binom{k}{2}} \leq 2^{c^2pn^2}$ choices for $\Gamma_0$, and at most $k^n \leq n^n = e^{n\log n}$ choices for $\phi$. So if $c$ is small enough, then by the union bound, with high probability we have $Z \geq \Omega(n^{1+1/g}) > \varepsilon n^2$ for every choice of $U,\Gamma_0,\phi$. If this happens, then no induced subgraph of $G$ on $n/2$ vertices has an $\varepsilon$-approximate homomorphism to an $F$-free graph on $k$ vertices, as required.  
\end{proof}
We remark that by taking blowups of $G'$ (constructed in the proof of the proposition), one can in fact find a graph witnessing this lower bound on $M_{F,H}(\varepsilon)$ with any number of vertices. Note that the upper and lower bounds on $M_{K_3,C_{\ell}}(\e)$ for odd $\ell \geq 9$ are off by a factor of roughly $\e^{1/2}$, and it would be interesting to close this gap.

\section{Concluding remarks}\label{sec:conclusion}
There remain a number of fascinating open problems related to asymmetric questions about graph homomorphisms. In particular, we again reiterate \cref{conj:no C7,conj:approx homo characterization}; the first claims that asymmetric homomorphism thresholds of odd cycles are already zero for $\delta_{\hom}(\C_{2t+3};\C_{2t+1})$, and the second would give a complete characterization of when one can get an exponential lower bound on $M_{F,H}(\varepsilon)$, showing that \cref{con:HG} is universal for such lower bounds.

While excluding \emph{all} odd cycles up to a certain length is very natural, the work of Sankar \cite{2206.07525} suggests that fascinating structure may arise if one studies, for example, $\delta_{\hom}(C_5;C_3)$ or $\delta_{\hom}(C_7;C_5)$. In particular, we would be very interested to learn whether Sankar's topological techniques can be used to give positive lower bounds on these quantities.

When it comes to asymmetric approximate homomorphisms, it would be interesting to tighten the lower and upper bounds appearing in \cref{thm:odd cycles}. While we obtain polynomial upper and lower bounds for $M_{K_3,C_\ell}(\varepsilon)$ for $\ell \geq 9$, and exponential upper and lower bounds for $M_{K_3,C_\ell}(\varepsilon)$ for $\ell \in \{5,7\}$, these bounds do not agree on the correct exponent for $\varepsilon$. 

\subsection*{Acknowledgments:} We are grateful to  Hong Liu and others for many helpful discussions on these topics. In particular, some of the ideas in the proof of \cref{thm:C3C5C7} were conceived in conjuction with Ant\'onio Gir\~ao, Freddie Illingworth, and Lukas Michel. Finally, we thank Sean English, Emily Heath, Andrew Simmons, and the anonymous referees for helpful comments on earlier drafts of this paper.
	\bibliographystyle{yuval}
	\bibliography{refs}

@book {lovasz_book,
    AUTHOR = {Lov\'{a}sz, L\'{a}szl\'{o}},
     TITLE = {Large networks and graph limits},
    SERIES = {American Mathematical Society Colloquium Publications},
    VOLUME = {60},
 PUBLISHER = {American Mathematical Society, Providence, RI},
      YEAR = {2012},
     PAGES = {xiv+475},
      ISBN = {978-0-8218-9085-1},
   MRCLASS = {05-02 (05C60 05C80 05C82 05D40)},
  MRNUMBER = {3012035},
MRREVIEWER = {Anant\ P.\ Godbole},
       DOI = {10.1090/coll/060},
       URL = {https://doi.org/10.1090/coll/060},
}

@article {GHIM,
    AUTHOR = {Gir\~ao, Ant\'onio and Hurley, Eoin and Illingworth, Freddie
              and Michel, Lukas},
     TITLE = {Abundance: asymmetric graph removal lemmas and integer
              solutions to linear equations},
   JOURNAL = {J. Lond. Math. Soc. (2)},
  FJOURNAL = {Journal of the London Mathematical Society. Second Series},
    VOLUME = {110},
      YEAR = {2024},
    NUMBER = {5},
     PAGES = {Paper No. e70015, 26pp},
      ISSN = {0024-6107,1469-7750},
   MRCLASS = {05C35 (11D04 11D45)},
  MRNUMBER = {4817874},
       DOI = {10.1112/jlms.70015},
       URL = {https://doi.org/10.1112/jlms.70015},
}

@article {GSW,
    AUTHOR = {Gishboliner, Lior and Shapira, Asaf and Wigderson, Yuval},
     TITLE = {An efficient asymmetric removal lemma and its limitations},
   JOURNAL = {Forum Math. Sigma},
  FJOURNAL = {Forum of Mathematics. Sigma},
    VOLUME = {13},
      YEAR = {2025},
     PAGES = {Paper No. e38, 17pp},
      ISSN = {2050-5094},
   MRCLASS = {05C35 (11B75)},
  MRNUMBER = {4861743},
       DOI = {10.1017/fms.2024.68},
       URL = {https://doi.org/10.1017/fms.2024.68},
}

@article {FZ,
    AUTHOR = {Fox, Jacob and Zhao, Yufei},
     TITLE = {Removal lemmas and approximate homomorphisms},
   JOURNAL = {Combin. Probab. Comput.},
  FJOURNAL = {Combinatorics, Probability and Computing},
    VOLUME = {31},
      YEAR = {2022},
    NUMBER = {4},
     PAGES = {721--736},
      ISSN = {0963-5483,1469-2163},
   MRCLASS = {05C35 (05D40 11B30)},
  MRNUMBER = {4439779},
MRREVIEWER = {Martin\ Klazar},
       DOI = {10.1017/s0963548321000572},
       URL = {https://doi.org/10.1017/s0963548321000572},
}

@article {FW_minimum_degree,
    AUTHOR = {Fox, Jacob and Wigderson, Yuval},
     TITLE = {Minimum degree and the graph removal lemma},
   JOURNAL = {J. Graph Theory},
  FJOURNAL = {Journal of Graph Theory},
    VOLUME = {102},
      YEAR = {2023},
    NUMBER = {4},
     PAGES = {648--665},
      ISSN = {0364-9024,1097-0118},
   MRCLASS = {05C15 (05C60)},
  MRNUMBER = {4563212},
MRREVIEWER = {Kiyoshi\ Yoshimoto},
}

@article {GJS_minimum_degree,
    AUTHOR = {Gishboliner, Lior and Jin, Zhihan and Sudakov, Benny},
     TITLE = {The minimum degree removal lemma thresholds},
   JOURNAL = {J. Combin. Theory Ser. B},
  FJOURNAL = {Journal of Combinatorial Theory. Series B},
    VOLUME = {166},
      YEAR = {2024},
     PAGES = {203--221},
      ISSN = {0095-8956,1096-0902},
   MRCLASS = {05C35},
  MRNUMBER = {4695634},
MRREVIEWER = {G\'abor\ N.\ S\'ark\"ozy},
       DOI = {10.1016/j.jctb.2024.01.003},
       URL = {https://doi.org/10.1016/j.jctb.2024.01.003},
}

@inproceedings {MR519318,
    AUTHOR = {Ruzsa, I. Z. and Szemer\'{e}di, E.},
     TITLE = {Triple systems with no six points carrying three triangles},
 BOOKTITLE = {Combinatorics ({P}roc. {F}ifth {H}ungarian {C}olloq.,
              {K}eszthely, 1976), {V}ol. {II}},
    SERIES = {Colloq. Math. Soc. J\'{a}nos Bolyai},
    VOLUME = {18},
     PAGES = {939--945},
 PUBLISHER = {North-Holland, Amsterdam-New York},
      YEAR = {1978},
   MRCLASS = {05C99 (05B30)},
  MRNUMBER = {519318},
MRREVIEWER = {W. G. Brown},
}

@article {MR1945375,
    AUTHOR = {Alon, Noga},
     TITLE = {Testing subgraphs in large graphs},
   JOURNAL = {Random Structures Algorithms},
  FJOURNAL = {Random Structures \& Algorithms},
    VOLUME = {21},
      YEAR = {2002},
    NUMBER = {3-4},
     PAGES = {359--370},
      ISSN = {1042-9832},
   MRCLASS = {05C85 (68R10)},
  MRNUMBER = {1945375},
       DOI = {10.1002/rsa.10056},
       URL = {https://doi-org.stanford.idm.oclc.org/10.1002/rsa.10056},
}

@incollection {MR3156927,
    AUTHOR = {Conlon, David and Fox, Jacob},
     TITLE = {Graph removal lemmas},
 BOOKTITLE = {Surveys in combinatorics 2013},
    SERIES = {London Math. Soc. Lecture Note Ser.},
    VOLUME = {409},
     PAGES = {1--49},
 PUBLISHER = {Cambridge Univ. Press, Cambridge},
      YEAR = {2013},
   MRCLASS = {05-02 (05C75)},
  MRNUMBER = {3156927},
MRREVIEWER = {Nikolaos Fountoulakis},
}

@article {MR1251840,
    AUTHOR = {Alon, N. and Duke, R. A. and Lefmann, H. and R\"{o}dl, V. and
              Yuster, R.},
     TITLE = {The algorithmic aspects of the regularity lemma},
   JOURNAL = {J. Algorithms},
  FJOURNAL = {Journal of Algorithms. Cognition, Informatics and Logic},
    VOLUME = {16},
      YEAR = {1994},
    NUMBER = {1},
     PAGES = {80--109},
      ISSN = {0196-6774},
   MRCLASS = {05C85 (68Q25 68R10)},
  MRNUMBER = {1251840},
MRREVIEWER = {Andrzej Ruci\'{n}ski},
       DOI = {10.1006/jagm.1994.1005},
       URL = {https://doi-org.stanford.idm.oclc.org/10.1006/jagm.1994.1005},
}

@inproceedings {MR1404036,
    AUTHOR = {F\"{u}redi, Zolt\'{a}n},
     TITLE = {Extremal hypergraphs and combinatorial geometry},
 BOOKTITLE = {Proceedings of the {I}nternational {C}ongress of
              {M}athematicians, {V}ol. 1, 2 ({Z}\"{u}rich, 1994)},
     PAGES = {1343--1352},
 PUBLISHER = {Birkh\"{a}user, Basel},
      YEAR = {1995},
   MRCLASS = {05C35 (05C65)},
  MRNUMBER = {1404036},
MRREVIEWER = {Sergei L. Bezrukov},
       DOI = {10.1007/978-3-0348-9078-6_65},
       URL = {https://doi-org.stanford.idm.oclc.org/10.1007/978-3-0348-9078-6_65},
}

@article {Csaba,
    AUTHOR = {Csaba, B\'ela},
     TITLE = {Regular decomposition of the edge set of a graph with
              applications},
   JOURNAL = {Australas. J. Combin.},
  FJOURNAL = {The Australasian Journal of Combinatorics},
    VOLUME = {89},
      YEAR = {2024},
     PAGES = {249--267},
      ISSN = {1034-4942,2202-3518},
   MRCLASS = {05C70 (05C85)},
  MRNUMBER = {4758570},
}

@article {MR2259060,
    AUTHOR = {Tao, Terence},
     TITLE = {A variant of the hypergraph removal lemma},
   JOURNAL = {J. Combin. Theory Ser. A},
  FJOURNAL = {Journal of Combinatorial Theory. Series A},
    VOLUME = {113},
      YEAR = {2006},
    NUMBER = {7},
     PAGES = {1257--1280},
      ISSN = {0097-3165,1096-0899},
   MRCLASS = {05C35 (05C65 05C75 11B75 37A45)},
  MRNUMBER = {2259060},
MRREVIEWER = {Jozef\ Skokan},
       DOI = {10.1016/j.jcta.2005.11.006},
       URL = {https://doi.org/10.1016/j.jcta.2005.11.006},
}

@article {MR4132523,
    AUTHOR = {Hoppen, Carlos and Kohayakawa, Yoshiharu and Lang, Richard and
              Lefmann, Hanno and Stagni, Henrique},
     TITLE = {Estimating parameters associated with monotone properties},
   JOURNAL = {Combin. Probab. Comput.},
  FJOURNAL = {Combinatorics, Probability and Computing},
    VOLUME = {29},
      YEAR = {2020},
    NUMBER = {4},
     PAGES = {616--632},
      ISSN = {0963-5483,1469-2163},
   MRCLASS = {68W20 (05C35 05C85 68R10)},
  MRNUMBER = {4132523},
       DOI = {10.1017/s0963548320000048},
       URL = {https://doi.org/10.1017/s0963548320000048},
}

@incollection {MR540024,
    AUTHOR = {Szemer\'edi, Endre},
     TITLE = {Regular partitions of graphs},
 BOOKTITLE = {Probl\`emes combinatoires et th\'eorie des graphes ({C}olloq.
              {I}nternat. {CNRS}, {U}niv. {O}rsay, {O}rsay, 1976)},
    SERIES = {Colloq. Internat. CNRS},
    VOLUME = {260},
     PAGES = {399--401},
 PUBLISHER = {CNRS, Paris},
      YEAR = {1978},
      ISBN = {2-222-02070-0},
   MRCLASS = {05C35 (10L10)},
  MRNUMBER = {540024},
MRREVIEWER = {D.\ A.\ Klarner},
}

@article {MR2811609,
    AUTHOR = {Fox, Jacob},
     TITLE = {A new proof of the graph removal lemma},
   JOURNAL = {Ann. of Math. (2)},
  FJOURNAL = {Annals of Mathematics. Second Series},
    VOLUME = {174},
      YEAR = {2011},
    NUMBER = {1},
     PAGES = {561--579},
      ISSN = {0003-486X,1939-8980},
   MRCLASS = {05C35 (11B75)},
  MRNUMBER = {2811609},
MRREVIEWER = {G\'abor\ N.\ S\'ark\"ozy},
       DOI = {10.4007/annals.2011.174.1.17},
       URL = {https://doi.org/10.4007/annals.2011.174.1.17},
}

@article {MR3939561,
    AUTHOR = {Moshkovitz, Guy and Shapira, Asaf},
     TITLE = {A sparse regular approximation lemma},
   JOURNAL = {Trans. Amer. Math. Soc.},
  FJOURNAL = {Transactions of the American Mathematical Society},
    VOLUME = {371},
      YEAR = {2019},
    NUMBER = {10},
     PAGES = {6779--6814},
      ISSN = {0002-9947,1088-6850},
   MRCLASS = {05C35 (05D99)},
  MRNUMBER = {3939561},
MRREVIEWER = {Nikolaos\ Fountoulakis},
       DOI = {10.1090/tran/7414},
       URL = {https://doi.org/10.1090/tran/7414},
}

@article {MR340075,
    AUTHOR = {Andr\'asfai, B. and Erd\H{o}s, P. and S\'os, V. T.},
     TITLE = {On the connection between chromatic number, maximal clique and
              minimal degree of a graph},
   JOURNAL = {Discrete Math.},
  FJOURNAL = {Discrete Mathematics},
    VOLUME = {8},
      YEAR = {1974},
     PAGES = {205--218},
      ISSN = {0012-365X,1872-681X},
   MRCLASS = {05C15},
  MRNUMBER = {340075},
MRREVIEWER = {D.\ J.\ Kleitman},
       DOI = {10.1016/0012-365X(74)90133-2},
       URL = {https://doi.org/10.1016/0012-365X(74)90133-2},
}

@article {MR169227,
    AUTHOR = {Andr\'asfai, B.},
     TITLE = {Graphentheoretische {E}xtremalprobleme},
   JOURNAL = {Acta Math. Acad. Sci. Hungar.},
  FJOURNAL = {Acta Mathematica. Academiae Scientiarum Hungaricae},
    VOLUME = {15},
      YEAR = {1964},
     PAGES = {413--438},
      ISSN = {0001-5954,1588-2632},
   MRCLASS = {55.10 (05.40)},
  MRNUMBER = {169227},
MRREVIEWER = {W.\ T.\ Tutte},
       DOI = {10.1007/BF01897150},
       URL = {https://doi.org/10.1007/BF01897150},
}

@incollection {MR671908,
    AUTHOR = {H\"aggkvist, Roland},
     TITLE = {Odd cycles of specified length in nonbipartite graphs},
 BOOKTITLE = {Graph theory ({C}ambridge, 1981)},
    SERIES = {North-Holland Math. Stud.},
    VOLUME = {62},
     PAGES = {89--99},
 PUBLISHER = {North-Holland, Amsterdam-New York},
      YEAR = {1982},
      ISBN = {0-444-86449-0},
   MRCLASS = {05C38},
  MRNUMBER = {671908},
}

@article {MR1264720,
    AUTHOR = {Jin, Guo Ping},
     TITLE = {Triangle-free graphs with high minimal degrees},
   JOURNAL = {Combin. Probab. Comput.},
  FJOURNAL = {Combinatorics, Probability and Computing},
    VOLUME = {2},
      YEAR = {1993},
    NUMBER = {4},
     PAGES = {479--490},
      ISSN = {0963-5483,1469-2163},
   MRCLASS = {05C38},
  MRNUMBER = {1264720},
MRREVIEWER = {J.\ Sedl\'a\v cek},
       DOI = {10.1017/S0963548300000845},
       URL = {https://doi.org/10.1017/S0963548300000845},
}

@article {MR342429,
    AUTHOR = {Erd\H{o}s, P. and Simonovits, M.},
     TITLE = {On a valence problem in extremal graph theory},
   JOURNAL = {Discrete Math.},
  FJOURNAL = {Discrete Mathematics},
    VOLUME = {5},
      YEAR = {1973},
     PAGES = {323--334},
      ISSN = {0012-365X,1872-681X},
   MRCLASS = {05C35},
  MRNUMBER = {342429},
MRREVIEWER = {D.\ R.\ Lick},
       DOI = {10.1016/0012-365X(73)90126-X},
       URL = {https://doi.org/10.1016/0012-365X(73)90126-X},
}

@article {MR2260851,
    AUTHOR = {\L{}uczak, Tomasz},
     TITLE = {On the structure of triangle-free graphs of large minimum
              degree},
   JOURNAL = {Combinatorica},
  FJOURNAL = {Combinatorica. An International Journal on Combinatorics and
              the Theory of Computing},
    VOLUME = {26},
      YEAR = {2006},
    NUMBER = {4},
     PAGES = {489--493},
      ISSN = {0209-9683,1439-6912},
   MRCLASS = {05C35},
  MRNUMBER = {2260851},
       DOI = {10.1007/s00493-006-0028-8},
       URL = {https://doi.org/10.1007/s00493-006-0028-8},
}

@article {MR1956996,
    AUTHOR = {Thomassen, Carsten},
     TITLE = {On the chromatic number of triangle-free graphs of large
              minimum degree},
   JOURNAL = {Combinatorica},
  FJOURNAL = {Combinatorica. An International Journal on Combinatorics and
              the Theory of Computing},
    VOLUME = {22},
      YEAR = {2002},
    NUMBER = {4},
     PAGES = {591--596},
      ISSN = {0209-9683,1439-6912},
   MRCLASS = {05C15 (05C35)},
  MRNUMBER = {1956996},
MRREVIEWER = {Mirko\ Hor\v n\'ak},
       DOI = {10.1007/s00493-002-0009-5},
       URL = {https://doi.org/10.1007/s00493-002-0009-5},
}

@unpublished{brandtthomasse,
  author = {Brandt, Stephan and Thomass\'e, St\'ephan},
  note = {Preprint available at \url{https://perso.ens-lyon.fr/stephan.thomasse/liste/vega11.pdf}},
  title = {Dense triangle-free graphs are four-colorable: {A} solution to the {Erd\H os--Simonovits} problem},
  year = {2011}
}

@article {MR3879964,
    AUTHOR = {Letzter, Shoham and Snyder, Richard},
     TITLE = {The homomorphism threshold of {$\{C_3,C_5\}$}-free graphs},
   JOURNAL = {J. Graph Theory},
  FJOURNAL = {Journal of Graph Theory},
    VOLUME = {90},
      YEAR = {2019},
    NUMBER = {1},
     PAGES = {83--106},
      ISSN = {0364-9024,1097-0118},
   MRCLASS = {05C60 (05C75)},
  MRNUMBER = {3879964},
MRREVIEWER = {Andr\'e\ E.\ K\'ezdy},
       DOI = {10.1002/jgt.22369},
       URL = {https://doi.org/10.1002/jgt.22369},
}

@article {MR4078811,
    AUTHOR = {Ebsen, Oliver and Schacht, Mathias},
     TITLE = {Homomorphism thresholds for odd cycles},
   JOURNAL = {Combinatorica},
  FJOURNAL = {Combinatorica. An International Journal on Combinatorics and
              the Theory of Computing},
    VOLUME = {40},
      YEAR = {2020},
    NUMBER = {1},
     PAGES = {39--62},
      ISSN = {0209-9683,1439-6912},
   MRCLASS = {05C35 (05C07 05C15 05D40)},
  MRNUMBER = {4078811},
       DOI = {10.1007/s00493-019-3920-8},
       URL = {https://doi.org/10.1007/s00493-019-3920-8},
}

@unpublished{2206.07525,
  author = {Sankar, Maya},
  note = {Preprint available at arXiv:2206.07525},
  title = {Homotopy and the homomorphism threshold of odd cycles},
  year = {2022}
}

@article {MR2321926,
    AUTHOR = {Thomassen, Carsten},
     TITLE = {On the chromatic number of pentagon-free graphs of large
              minimum degree},
   JOURNAL = {Combinatorica},
  FJOURNAL = {Combinatorica. An International Journal on Combinatorics and
              the Theory of Computing},
    VOLUME = {27},
      YEAR = {2007},
    NUMBER = {2},
     PAGES = {241--243},
      ISSN = {0209-9683,1439-6912},
   MRCLASS = {05C15},
  MRNUMBER = {2321926},
       DOI = {10.1007/s00493-007-0054-1},
       URL = {https://doi.org/10.1007/s00493-007-0054-1},
}

@article {MR69494,
    AUTHOR = {Mycielski, J.},
     TITLE = {Sur le coloriage des graphs},
   JOURNAL = {Colloq. Math.},
  FJOURNAL = {Colloquium Mathematicum},
    VOLUME = {3},
      YEAR = {1955},
     PAGES = {161--162},
      ISSN = {0010-1354,1730-6302},
   MRCLASS = {56.0X},
  MRNUMBER = {69494},
MRREVIEWER = {W.\ T.\ Tutte},
       DOI = {10.4064/cm-3-2-161-162},
       URL = {https://doi.org/10.4064/cm-3-2-161-162},
}

@article {MR1192374,
    AUTHOR = {Hell, Pavol and Ne\v{s}et\v{r}il, Jaroslav},
     TITLE = {The core of a graph},
   JOURNAL = {Discrete Math.},
  FJOURNAL = {Discrete Mathematics},
    VOLUME = {109},
      YEAR = {1992},
    NUMBER = {1-3},
     PAGES = {117--126},
      ISSN = {0012-365X},
   MRCLASS = {05C75 (05C85 68Q25)},
  MRNUMBER = {1192374},
MRREVIEWER = {Gary MacGillivray},
       DOI = {10.1016/0012-365X(92)90282-K},
       URL = {https://doi-org.stanford.idm.oclc.org/10.1016/0012-365X(92)90282-K},
}

@article {MR1723039,
    AUTHOR = {Frieze, Alan and Kannan, Ravi},
     TITLE = {Quick approximation to matrices and applications},
   JOURNAL = {Combinatorica},
  FJOURNAL = {Combinatorica. An International Journal on Combinatorics and
              the Theory of Computing},
    VOLUME = {19},
      YEAR = {1999},
    NUMBER = {2},
     PAGES = {175--220},
      ISSN = {0209-9683},
   MRCLASS = {68Q25 (05C85 15-04 68R10 68W20)},
  MRNUMBER = {1723039},
MRREVIEWER = {Eugene V. Dulov},
       DOI = {10.1007/s004930050052},
       URL = {https://doi.org/10.1007/s004930050052},
}

@unpublished{2501.15861,
  title={Many pentagons in triple systems},
  author={Mubayi, Dhruv and Solymosi, Jozsef},
  note={Preprint available at arXiv:2501.15861},
  year={2025}
}

@unpublished{2502.09576,
  title={Interpolating chromatic and homomorphism thresholds},
  author={Huang, Xinqi and Liu, Hong and Rong, Mingyuan and Xu, Zixiang},
  note={Preprint available at arXiv:2502.09576},
  year={2025}
}

@phdthesis{Stiebitz,
    author = {Stiebitz, Michael},
    title = {Beitr\"age zur {T}heorie der f\"arbungskritischen {G}raphen},
    school = {TU Ilmenau},
    year = {1985}
}

@article {MR884223,
    AUTHOR = {Haussler, David and Welzl, Emo},
     TITLE = {{$\epsilon$}-nets and simplex range queries},
   JOURNAL = {Discrete Comput. Geom.},
  FJOURNAL = {Discrete \& Computational Geometry. An International Journal
              of Mathematics and Computer Science},
    VOLUME = {2},
      YEAR = {1987},
    NUMBER = {2},
     PAGES = {127--151},
      ISSN = {0179-5376,1432-0444},
   MRCLASS = {68U05 (52-04 68P20)},
  MRNUMBER = {884223},
       DOI = {10.1007/BF02187876},
       URL = {https://doi.org/10.1007/BF02187876},
}

@unpublished{Galvin,
  title={Three tutorial lectures on entropy and counting},
  author={Galvin, David},
  note={Preprint available at arXiv:1406.7872},
  year={2014}
}

	\appendix

	\section{Lemmas from Section \ref{sec:approx hom hard}}\label{sec:F-tree}
	   
	\begin{proof}[Proof of \cref{lem:hyperforest}]
		The ``only if" direction is easy to prove by induction on the number of edges. We prove the ``if" direction. The proof is by induction on $e(\mathcal{G})$, and the base case $e(\mathcal{G}) = 0$ is trivial, so suppose that $e(\mathcal{G}) \geq 1$. Consider the bipartite graph with sides $A = V(\mathcal{G})$ and $B = E(\mathcal{G})$ where $v$ and $e$ are adjacent if $v \in e$. A cycle in this graph gives a Berge cycle in $\mathcal{G}$, so this graph is a forest. 
		Note that each $e \in B$ has degree $|e| \geq 2$ in the auxiliary graph, so a path ending in $e \in B$ can always be extended. 
		We claim that there is $e \in B$ which has at most one non-leaf neighbour. Indeed, take a longest path $P$ starting in $A$, let $v \in A$ be the first vertex of $P$, and let $e \in B$ be the neighbour of $v$ on $P$.  
		Then by the maximality of $P$, $e$ has at most one non-leaf neighbour. This means that in $\mathcal{G}$, $e$ has at most one vertex which is contained in another edge of $\mathcal{G}$. By induction, $\mathcal{G} - e$ is a hyperforest. Adding back $e$ preserves this property.
	\end{proof}

	\begin{proof}[Proof of \cref{lem:high-girth hypergraph}]
		Consider a random $f$-uniform, $f$-partite hypergraph with parts of size $n/f$ and edge probability $p = cn^{-f+1+1/(g-1)}$, where $c > 0$ is a small constant to be chosen later. The expected number of edges is {$p(n/f)^f = \Omega(n^{1+1/(g-1)})$.} For each $2 \leq k \leq g$, the expected number of Berge cycles of length $k$ is at most 
		$
		n^k n^{(f-2)k} p^k = n^{(f-1)k}p^k \leq \frac{1}{2g} \cdot p(n/f)^f,
		$ 
		using our choice of $p$ and provided that $c$ is a small enough constant. 
		So the expected number of Berge cycles of length at most $g$ is at most $\frac{1}{2}p(n/f)^f$.
		Deleting one edge from each Berge cycle of length at most $g$ gives the result. 
	\end{proof}

\end{document}